\documentclass[a4paper, 12pt]{amsart}
\usepackage{amscd}
\usepackage{amsmath, amsthm, amssymb}
\usepackage{t1enc}
\usepackage[mathscr]{eucal}
\usepackage{graphicx}
\usepackage{graphics}
\usepackage{pict2e}
\usepackage{epic}
\numberwithin{equation}{section}
\usepackage[margin=2.9cm, marginparwidth=2.5cm, marginparsep=.5cm]{geometry}
\usepackage{marginnote}
\usepackage{epstopdf} 
\usepackage{xcolor}
\usepackage{commath}

\usepackage{tikz-cd}
\usepackage{mathabx}
\usepackage{hyperref}
\hypersetup{
	colorlinks,
	linkcolor={blue!50!black},
	citecolor={red!50!black},
	urlcolor={blue!80!black}
}
\usepackage{enumitem}
\usepackage{geometry}
\usepackage{stackengine}
\usepackage{overpic}
\usepackage[rightcaption]{sidecap}
\usepackage{caption}
\usepackage{subcaption}
\graphicspath{ {images/} }
\usepackage{fancyhdr}
\theoremstyle{plain}

\newtheorem*{theorem*}{Theorem}
\newtheorem{theorem}{Theorem}[section]
\newtheorem{proposition}[theorem]{Proposition}

\newtheorem{lemma}[theorem]{Lemma}
\newtheorem{observation}[theorem]{Observation}

\newtheorem{corollary}[theorem]{Corollary}
\newtheorem{property}[theorem]{Property}

\newtheorem*{claim 0}{Claim 0}
\newtheorem*{claim 1}{Claim 1}
\newtheorem*{claim 2}{Claim 2}
\newtheorem*{claim 3}{Claim 3}
\newtheorem{definition}[theorem]{Definition}

\theoremstyle{definition}

\newtheorem{assumption}[theorem]{Assumption}

\newcommand{\RR}{\mathbb{R}}
\newcommand{\NN}{\mathbb{N}}

\newcommand{\SC}{S^{1}}

\newcommand{\G}{\mathcal{G}}
\newcommand{\GG}{\widetilde{\mathcal{G}}}
\newcommand{\F}{\mathcal{F}}
\newcommand{\FF}{\widetilde{\mathcal{F}}}

\newcommand{\A}{\mathcal{A}}
\newcommand{\R}{\mathcal{R}}

\newcommand{\M}{\mathcal{M}}
\newcommand{\CCC}{\mathcal{C}}
\newcommand{\MM}{\widetilde{\mathcal{M}}}

\newcommand{\U}{\mathcal{U}}
\newcommand{\V}{\mathcal{V}}
\newcommand{\TPhi}{\widetilde{\Phi}}
\newcommand{\wwp}{\widetilde \Phi}
\newcommand{\zein}{[0,\infty)}
\newcommand{\wT}{\widetilde T}

\usepackage{mathtools}

\begin{document}
	
	\title{Leafwise Quasigeodesic Foliations in Dimension Three And the Funnel Property }
	
	\begin{center}
		
		\author{ANINDYA CHANDA AND S\'ERGIO FENLEY}

	\end{center}
	
	\date{}
	\subjclass[2020]{
		Primary: 37C10, 37D20, 57R30, 53C23.
		Secondary: 37D40, 37D45, 37C15, 37D10.}
	
	\keywords{Quasigeodesics, subfoliations, Anosov flows.}
	
	\address{Florida State University, Tallahassee, FL 32304, United States}
	\email{ac17t@my.fsu.edu/achanda@math.fsu.edu}
	\address{Florida State University, Tallahassee, FL 32304, United States}
	\email{fenley@math.fsu.edu }

	\begin{abstract} 
		We construct one dimensional foliations which are
		subfoliations of two dimensional foliations in
		$3$-manifolds.
		The subfoliation is by quasigeodesics in each
		two dimensional leaf, but it is not funnel: 
		not all quasigeodesics share  a common ideal point
		in most leaves.
	\end{abstract}
	\maketitle
	
	
	\section{Introduction}

	The goal of this article is to analyze whether certain geometric
	conditions imply that a one dimensional foliation
	in a $3$-manifold is the foliation by flow lines of a topological
	Anosov flow. We do this analysis for one dimensional foliations
	whose leaves lie inside leaves of two dimensional foliations
	and whose leaves are quasigeodesics in these two dimensional
	foliations. In other words the goal of this article is to analyse
	whether some strictly 
	geometric behavior implies strong dynamical systems behavior
	in this setting.
	This has important connections with partial hyperbolicity
	in dimension $3$.

	A foliation $\G$ subfoliates a foliation $\F$ if each
	leaf of $\F$ has a foliation made up of leaves of $\G$.
	We call $\G$ the subfoliation and $\F$ the super foliation.
	This situation is very common, for example if $\F_1$ and
	$\F_2$ are two foliations which are transverse to each
	other everywhere,
	then their intersection forms a subfoliation of each
	of them.
	This article aims to study geometric properties of 
	leaves of subfoliations inside the leaves of the 
	super foliation.

	One very common and extremely important example is the
	following: let $\Phi$ be an Anosov flow and let
	$\F^{ws}, \F^{wu}$ be the weak stable and weak unstable foliations of $\Phi$
	\cite{Ano63, KH95}. Then $\F^{ws}, \F^{wu}$
	are transverse to each other $-$ the intersection is the
	foliation by flow lines of $\Phi$ which is a subfoliation
	of each of them. 
	This example has connections with geometry or large scale
	geometry: the leaves of $\F^{ws}, \F^{wu}$ are Gromov hyperbolic.
	In rough terms this means that they are negatively curved.
	The subfoliation by flow lines in, say $\F^{ws}$,
	satisfies an additional strong geometric property: in each leaf of
	$\F^{ws}$ the flow lines are quasigeodesics. This means that
	when lifted to the universal cover of the leaves,
	the flow lines are uniformly
	efficient up to a bounded multiplicative distortion in 
	measuring length in the weak stable leaves.
	In other words the flow lines are quasi-isometrically embedded in
	these weak stable leaves. The quasigeodesic property has
	many important consequences, for example the flow lines
	are within bounded distance from length minimizing geodesics
	when lifted to the universal cover of their respective weak stable 
	leaves
	(\cite{Thu82, Thu97, Gro87}). Hence the flow lines have well defined
	distinct ideal points in the Gromov boundary of the weak stable 
	leaves in both
	directions. 
	These properties and others are very strong and useful in
	many contexts.
	Obviously this also works for the flow subfoliation of the weak 
	unstable foliation.
	
	A (one dimensional) subfoliation made of quasigeodesics in the leaves of
	a super foliation by Gromov hyperbolic
	leaves is called a {\em {leafwise quasigeodesic foliation}}.
	
	The Anosov case has an additional geometric property: in (say)
	a weak stable leaf all flow lines are forward asymptotic, this is a
	defining property of the weak stable foliation. In particular
	all flow lines in a given weak stable leaf have the same
	forward ideal point in the ideal boundary of the weak stable
	leaf (when lifted to the universal cover). 
	
	When all leaves of a leafwise quasigeodesic 
	subfoliation in a leaf of the super foliation have a common 
	ideal point we call that leaf a {\em {funnel leaf}}. If
	all leaves of the super foliation are funnel leaves then
	the leafwise quasigeodesic foliation is said to have
	the {\em {funnel property}}.
	
	The motivation for this article is the following question: is the
	funnel property an additional property or is it a consequence
	of the leafwise quasigeodesic property?
	The importance of this is the following: in dimension $3$ we
	have a much stronger connection between some of these properties
	as follows. Suppose that $\G$ is a leafwise quasigeodesic
	foliation (which is a one dimensional subfoliation
	of a two dimensional foliation) which has the funnel property. 
	The ambient manifold is $3$-dimensional.
	Suppose that the foliation $\G$ is orientable, or in other words
	it is the foliation of a non singular flow.
	Then one can prove
	(we refer to \cite{BFP20} for definitions of the terms used here and
	for detailed proofs) 
	that the flow in question is
	expansive. 
	This implies that the flow is a topological Anosov flow
	(\cite[Theorem 15.]{IM90}, \cite[Lemma 7]{Pat93}).
	If the flow is transitive (the union of periodic
	orbits is dense) then the topological Anosov flow is in
	addition orbitally equivalent to a (smooth) Anosov flow
	(\cite{Sha20}).
	This means that if the leafwise quasigeodesic property
	implies the funnel property, then this purely geometric condition
	implies a very strong dynamical systems property: the
	foliation is the flow foliation of an Anosov flow, up
	to topological equivalence..
	
	In this article we prove that the funnel property is
	not a consequence of leafwise quasigeodesic behavior:
	
	\begin{theorem}\label{main}
		There are examples of leafwise quasigeodesic
		foliations in dimension $3$ which do not have the funnel
		property.
	\end{theorem}

	\vskip .05in
	We now briefly explain 
	one class of examples: start with the Franks-Williams
	example of a non transitive Anosov flow $\Phi$. 
	This is obtained as follows: start with a suspension Anosov flow
	and do a DA (derived from Anosov) blow up of a periodic orbit transforming
	it into (say) a repelling orbit $\alpha$. 
	Remove a tubular neighborhood of $\alpha$ so that the resulting semiflow
	is incoming in the complement of the removed 
	tubular neighborhood of $\alpha$. Glue this manifold with 
	boundary with a copy of it which has a reversed
	flow. One fundamental result is that the ensuing flow $\Phi$ in the
	final manifold $\M$ is Anosov
	\cite{FW80,BBY17}. 
	This holds for certain isotopy classes of gluings and
	certain gluing maps satisfying transversality conditions.
	These were the first examples of
	non transitive Anosov flows in dimension $3$.
	Our examples use this flow.
	There is a smooth torus
	$T$ in $\M$ transverse to the flow. 
	There is a single two dimensional attractor and a single two
	dimensional repeller of the flow $\Phi$ in $M$.
	Start with a one dimensional foliation
	$Z$ in $T$ which is transverse to the intersections of both the
	weak stable and the weak unstable foliations of $\Phi$ with $T$.
	Saturate $Z$ by the flow producing a collection of two 
	dimensional sets embedded in $\M$. The flow saturation of
	$T$ is an open subset $V$ of $M$, and the collection of the
	two dimensional subsets described is a two dimensional
	foliation in $V$. In addition $V$
	is exactly the complement of the union of the
	attractor and the repeller of $\Phi$. Complete the foliation in $V$
	to a foliation
	$\F$ in $\M$ which is the weak unstable
	foliation of $\Phi$ in the attractor of $\Phi$ 
	and the weak stable foliation
	in the repeller of $\Phi$. 
	The proof that this is in fact a foliation of $\M$ 
	depends on a careful choice of the one
	dimensional foliation $Z$ in $T$. There is a subtle point here in
	that if one chooses an arbitrary foliation $Z$ in $T$,  then when
	lifting to $\MM$ the lifted sets may not be properly embedded in $\MM$
	and so $\F$ would not be a foliation in $\M$. 
	This is carefully analyzed in section \ref{s.foliation} and
	there we prove that for appropriate choices of $Z$ the object
	$\F$ we construct is a foliation.
	The super foliation is this two dimensional
	foliation $\F$. 
	The subfoliation $\G$ of $\F$ 
	is the foliation by flow lines of $\Phi$.
	Each leaf of $\F$ is saturated by flow lines. 
	We prove that $\G$ is a leafwise quasigeodesic subfoliation
	of $\F$, but $\G$ does not have the funnel property.
	There is an Anosov flow $\Phi$ in this example, however notice that
	the super foliation $\F$ is neither the weak stable nor
	the weak unstable foliation of $\Phi$, but rather a different foliation.
	In fact in the same way one can construct an infinite number of
	inequivalent examples with the same starting flow $\Phi$. The foliations are pairwise
	distinguished because of how they intersect the torus
	$T$ in foliations which are not equivalent.

	In this article we 
	consider more general examples. We prove that one
	can construct examples starting with any non transitive
	Anosov flow $\Phi$ in dimension $3$  so that all the basic sets
	have dimension $2$. As in the case of the
	Franks-Williams example, we construct super foliations which have
	Gromov hyperbolic leaves and whose leaves are saturated
	by flow lines of $\Phi$. We show that the 
	subfoliation $\G$ by flow lines of $\Phi$
	is by quasigeodesics in each leaf of the super foliation $\F$.
	This is the hardest step to prove.
	This involves a very careful analysis of the geometry in these
	examples. The proof that $\G$ is not funnel is simpler than
	proving it is leafwise quasigeodesic as a subfoliation of $\F$.
	
	We finish this introduction mentioning another reason why we analyzed
	this question: this comes from partially hyperbolic dynamics.
	Let $f$ be a partially hyperbolic diffeomorphism in 
	a closed $3$-manifold $M$ (we refer to \cite{BFP20} for definitions
	and properties of partially hyperbolic diffeomorphisms).
	Under very general orientability conditions, there is a pair
	of transverse two
	dimensional branching foliations (center stable and
	center unstable foliations) associated with the partially
	hyperbolic diffeomorphism which intersect in an one dimensional
	branching foliation, called the center foliation
	\cite{BI08}. The center
	foliations subfoliates both the center stable and center unstable
	foliations.
	In some situations (\cite{BFP20}) it is shown
	that the center foliation is a leafwise quasigeodesic subfoliation
	of both the center stable and center unstable foliations.
	But in \cite{BFP20} it is proved that 
	in the partially hyperbolic setting the leafwise quasigeodesic
	property implies 
	that the center foliation has the funnel property (as a subfoliation
	of both super foliations). The proof of this also uses dynamical
	systems properties, namely partial hyperbolicity. 
	An open question from the article \cite{BFP20}
	was to whether the funnel 
	property could be derived strictly from the leafwise quasigeodesic
	property in (say) the center stable foliation. 
	In this article we prove that this is not the case, by constructing
	counterexamples for general foliations.
	
	\textbf{Acknowledgement:} We thank Rafael Potrie for providing a crucial idea which 
	greatly simplified the proof of Lemma \ref{lem.emb}.

	\section{Preliminaries}
	
	\par A $C^1$-flow $\Phi_t:\M\rightarrow\M$ on a Riemannian manifold $\M$ is $Anosov$ if the tangent bundle $T\M$ splits into three $D\Phi_{t}$-invariant sub-bundles $T\M=E^{s}\oplus E^{0} \oplus E^u$ and there exists two constants $C,\lambda>0$ such that 
	\begin{itemize}
		\item $E^0$ is generated by the non-zero vector field defined by the flow $\Phi_t$;
		\item For any $v\in E^s$ and $t>0$, $$||D\Phi_{t}(v)||\leq Ce^{-\lambda t}||v||$$
		\item For any $w\in E^u$ and $t>0$, $$||D\Phi_{t}(w)||\geq Ce^{\lambda t}||w||$$
	\end{itemize}
	
	The definition is independent of the choice of the Riemannian metric $||.||$ as the underlying manifold $\M$ is compact. For a point $x\in \M$, the set $\gamma_x=\{\Phi_t(x)|t\in \RR\}$ is called the \textit{flow line} of $x$. The collection of all flow lines of a flow defines a one-dimensional  foliation on $\M$. For an Anosov flow there are several flow invariant foliations associated to the flow and these foliations play a key role in the study of Anosov flows. \\

	\begin{property}[\cite{Ano63}]
		For an Anosov flow $\Phi_t$ on $\M$, the distributions $E^u$, $E^s$, $E^{0}\oplus E^{u}$ and $E^{0}\oplus E^{s}$ are uniquely integrable. The associated foliations are denoted by $\F^u$, $\F^s$, $\F^{wu}$ and $\F^{ws}$ 
		respectively and they are called the strong unstable, strong stable, weak unstable and weak stable foliation on $\M$. 
	\end{property}
	
	\noindent
	{\bf {For the remainder of this article we will assume that $\M$ is a closed three dimensional Riemannian manifold.}}
	
	\vskip .05in
	We also assume that $\M$ is 
	equipped with an Anosov flow $\Phi_t$ and $\widetilde{\Phi}_t$ 
	is the lift of the flow $\Phi_t$ in $\MM$, the universal cover of $\M$. The strong unstable, strong stable, weak unstable and weak stable foliation of $\TPhi$ are the lifts of  the foliations $\F^u$, $\F^s$, $\F^{wu}$ and $\F^{ws}$ in the universal cover $\MM$, and these foliations in $\MM$ are denoted by $\FF^u$, $\FF^s$, $\FF^{wu}$ and $\FF^{ws}$ respectively. 
	
	
	A map $f:(X_1,d_1)\rightarrow (X_2,d_2)$ is called a $(K,s)$\textit{-quasi-isometric embedding} if there exits $K>1$ and $s>0$ such that for all $x,y\in X_1$
	$$\frac{1}{K}d_1(x,y)-s\leq d_2 (f(x),f(y)) \leq K d_1(x,y)+s $$
	A $(K,s)$-\textit{quasigeodesic} in $X$ is the image of a $(K,s)$-\textit{quasi-isometric embedding}\\ $\gamma:[a,b]\rightarrow X$ where $[a,b]$ is a closed 
	interval on $\RR$ with the Euclidean metric.
	The interval could be infinite (that is $a = -\infty$, $b = \infty$ or both), in which case the notation would be of a half open or open interval.
	If we have a map $\RR \rightarrow X$ with rectifiable image we consider
	the arclength metric in the domain $\RR$.
	
	\begin{lemma}\label{lqg}
		Flow lines on the leaves in $\FF^{ws}$ and $\FF^{wu}$ are quasigeodesics with respect to the metric induced from $\MM$ in their respective leaves. 
	\end{lemma}
	
	\begin{proof}
		Reparametrize the flow to have unit speed. The new flow
		is still Anosov \cite{Ano63} with the same flow lines and the
		same weak stable and weak unstable foliations. However the
		strong stable and strong unstable leaves may change.
		
		Any leaf $L$  of $\FF^{wu}$ is subfoliated by $\FF^{u}$ and by the flow lines, these two foliations are transversal to each other. We can define a metric $ds'$ on $L$ by 
		$ds'=dw+dy$ where $dw$ measures length along flow lines and $dy$ measures length along unstable curves. Suppose $ds$ is the Riemannian 
		metric induced on $L^{wu}$ from $\MM$. The two path metrics induced
		in $L$ from $ds'$ and $ds$ on are uniformly
		quasi-isometric to each other \cite{Fen94}. Moreover each flow line in the leaf $L$ is a length minimizing
		curve in the $ds'$ metric, hence flow lines
		are uniform quasigeodesics with respect to 
		the metric induced by $ds$. Similarly it can be shown that flow lines on leaves in $\FF^{ws}$ are quasigeodesic with respect to the induced metric on their respective leaves. 
	\end{proof}
	
	\begin{definition}
		Suppose $\F$ is a two dimensional foliation on $\M$ with
		Gromov hyperbolic leaves when lifted to the universal cover. Suppose that $\G$ is a one dimensional foliation on $\M$ which subfoliates $\F$.
		In this situation we say that leaves of $\G$ are leafwise quasigeodesic in $\F$ if every leaf of $\G$ is a quasigeodesic in the respective leaf of $\F$ containing it when lifted to the universal cover of the leaf.   
		In that case we say that $\G$ is a leafwise quasigeodesic subfoliation
		of $\F$.
	\end{definition}

	In Lemma \ref{lqg} the flow lines of $\Phi_t$ are shown to
	be \textit{leafwise quasigeodesics} in the leaves of $\F^{ws}$ and $\F^{wu}$.
	
	The leaves in $\FF^{ws}$ and $\FF^{wu}$ are Gromov hyperbolic with respect to the Riemannian metric on the leaves induced from the metric on $\MM$ \cite{Fen94}. Suppose that $L$ is a leaf either in $\FF^{ws}$ or in $\FF^{wu}$. 
	As the leaves are Gromov hyperbolic, we can define the ideal boundary of $L$ which is homeomorphic to the circle and we denote it as $S^1(L)$. 
	The compactification $L \cup S^1(L)$ is homeomorphic to a 
	closed disk.
	As the flow lines are quasigeodesics in $L$, they define two distinct
	ideal points on $S^1(L)$: If $\gamma$ is a flow line in $L$ then the forward ray of $\gamma$ defines an unique ideal point on $S^1(L)$ as $\gamma$ is a quasigeodesic, which is called the
	\textit{forward or positive ideal point} of $\gamma$. Similarly we 
	define the \textit{backward or negative ideal point} 
	as the limit of the ray in the backwards direction.
	The following statement describes the equivalence between the forward and backward flow rays in the leaves of $\FF^{ws}$ and $\FF^{wu}$ and the points on their ideal boundaries.
	
	\begin{figure}
		\begin{subfigure}[b]{0.45\textwidth}
			\centering
			\includegraphics[width=\textwidth]{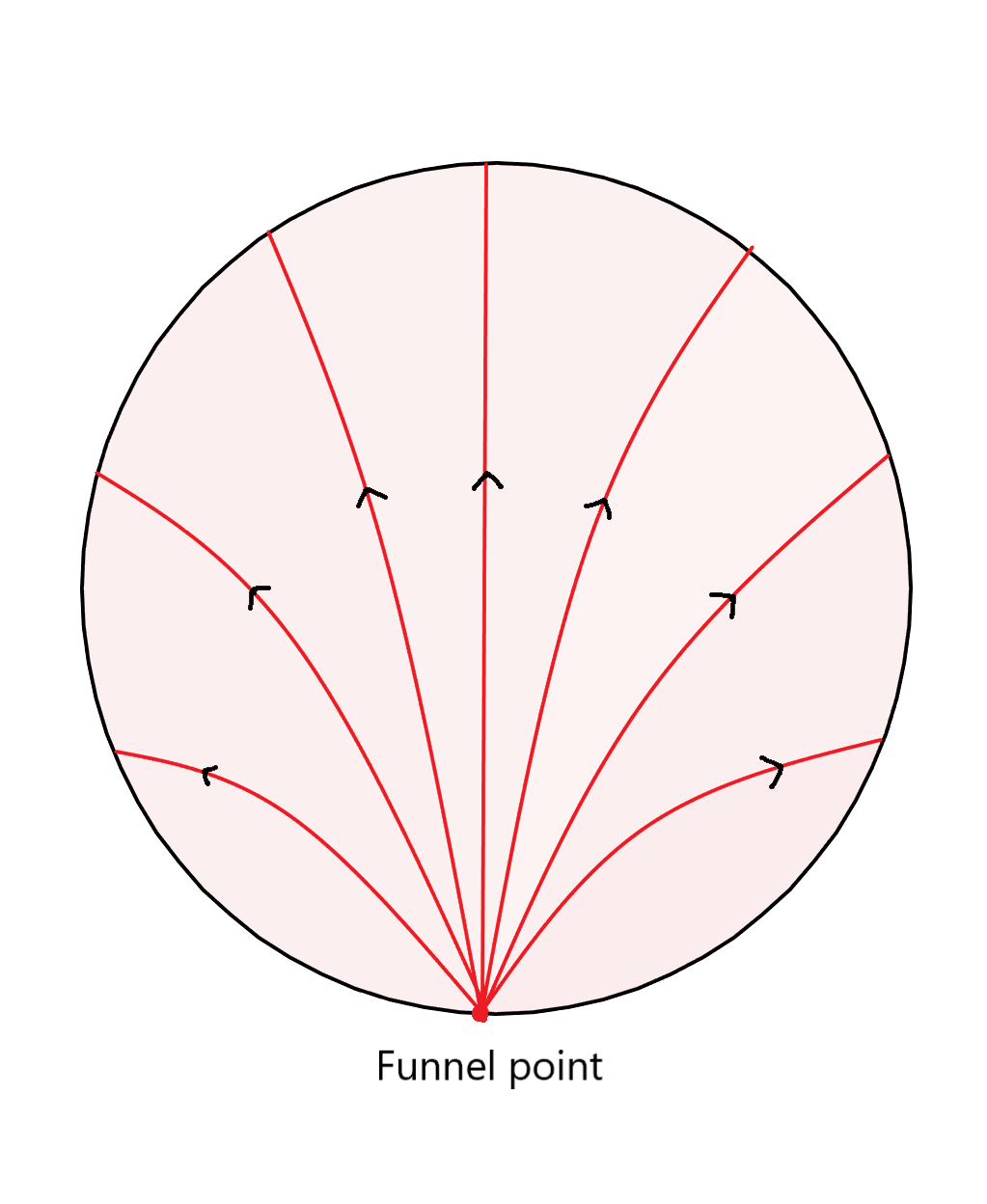}
			\caption{leaves in $\widetilde{\F}^{wu}$}
			\label{fig:y equals x}
		\end{subfigure}
		\hfill
		\begin{subfigure}[b]{0.45\textwidth}
			\centering
			\includegraphics[width=\textwidth]{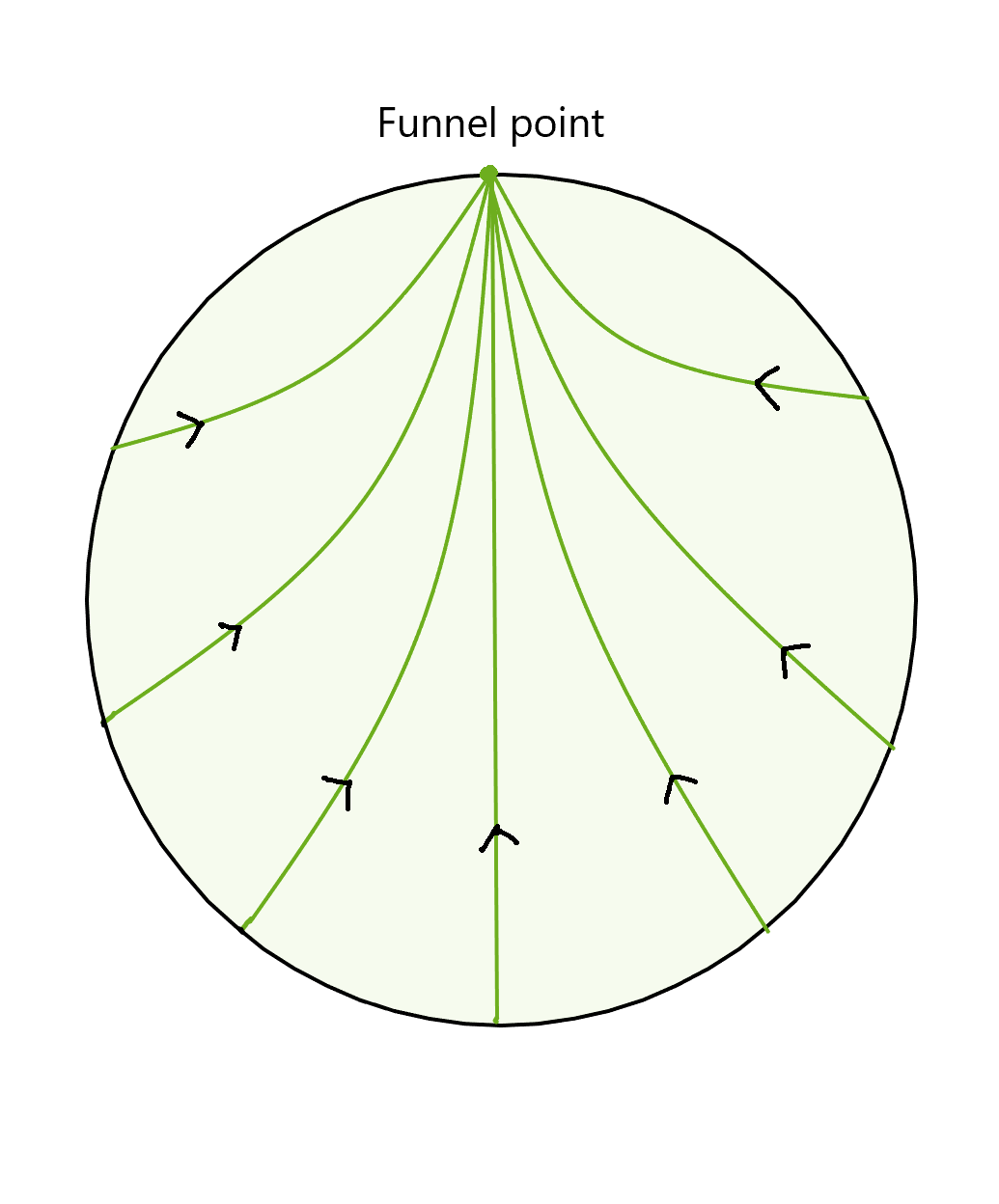}
			\caption{Leaves in $\widetilde{\F}^{ws}$}
			\label{fig:three sin x}
		\end{subfigure}
		
		\caption{Geometry of flow lines on the leaves in $\widetilde{\F}^{wu}$ and $\widetilde{\F}^{wu}$ }
		\label{UAS}
	\end{figure}
	
	\begin{property}[\cite{Fen94}]{\label{2.4}}
		For a  leaf $L$ either in $\FF^{ws}$ or $\FF^{wu}$, all the points on $S^1(L)$ correspond to forward or backward flow rays on $L$. If $L\in\FF^{ws}$ then all the flow lines on $L$ have a common forward ideal point and all the other ideal points are backward ideal points on $S^1(L)$ of the flow lines. No two different flow lines define a common negative or backward ideal point. 
		
		If $L\in\FF^{wu}$ then all the flow lines have a common backward ideal point and all the forward flow lines defines all the other ideal points on $S^1(L)$. No two different flow lines define the same positive or forward ideal point.  
	\end{property}
	
	The property for forward ideal points in $\FF^{ws}$ is immediate as
	these flow lines are forward asymptotic, a direct consequence of the definitions. The property for backward 
	ideal points in leaves of $\FF^{ws}$ is not as immediate and is proved
	in \cite{Fen94}.
	
	\begin{definition}
		Suppose that $\G$ is a leafwise quasigeodesic subfoliation of $\F$.
		If a leaf $L$ of $\FF$ has all leaves of $\GG$ in it sharing
		a common ideal point then the projected leaf $\pi(L)$ of $\F$ in $\M$
		is called a \textit{funnel leaf}. 
		In this case the common ideal points shared by all the flow lines in $L$ is called the \textit{funnel point} of $L$.
	\end{definition}
	
	\begin{corollary}\label{f}
		By property \ref{2.4}, for an Anosov flow $\Phi_t$ on a three manifold $\M$, with the flow foliation a leafwise 
		quasigeodesic subfoliaton
		of both $\F^{ws}$ and $\F^{wu}$ the following happens:
		all the leaves in weak stable foliation $\F^{ws}$ and weak unstable foliation $\F^{wu}$ are funnel leaves as shown in figure \ref{UAS}.
	\end{corollary}

	\noindent
	\textbf{Basic sets of Anosov flows on three manifolds:} The Anosov flow $\Phi$ is called $transitive$ if there exists a flow line $\gamma$ dense in $\M$, otherwise the flow is \textit{non transitive}. The first example of a non transitive Anosov flow was constructed by John Franks and Bob Williams in their 1980's article \cite{FW80}.
	
	\par A point $x\in\M$ is called \textit{nonwandering} if for any open neighborhood $U$ of $x$ and any $t_0 > 0$,
	there exists $t> t_0$ such that $\Phi_{t}(U)\cap U\neq \emptyset$, the set of all nonwandering points is denoted by $\Omega(\Phi)$. For a non transitive Anosov flow $\Phi_t$ the nonwandering set $\Omega(\Phi)$ is not equal to the whole manifold $\M$ and according to \textit{Spectral Decomposition Theorem} \cite{Sma67}, $\Omega(\Phi)$
	is decomposed into finitely many closed, disjoint, $\Phi_t$-invariant and transitive $basic$ sets $\{ \Lambda_i, i = 1, ..., n\}$, so
	$\Omega(\Phi)=\bigsqcup\limits_{i=1}^{n} \Lambda_i$. 
	
	Suppose $\Lambda$ is a basic set of a non transitive Anosov flow $\Phi_t$ on a three manifold. Then $\Lambda$ can be characterised in four different types \cite{ Sma67, Bru93},
	\begin{itemize}
		\item $dim(\Lambda)=2$ and  the basic set $\Lambda$ is an \textit{attractor}, i,e. there exists an open set U containing $\Lambda$ such that $\bigcap\limits_{t>0}\Phi_t(U)=\Lambda$.
		\item $dim(\Lambda)=2$ and  the basic set $\Lambda$ is a \textit{repeller}, i,e $\Lambda$ is an attractor for the reversed flow $\Psi_t=\Phi_{-t}$.
		\item $dim(\Lambda)=1$ and $\Lambda$ is a saddle with local cross
		section a Cantor set.
		\item $dim(\Lambda)=1$ and $\lambda$ is a hyperbolic periodic orbit. 
	\end{itemize}
	
	\begin{property}
		If $\Lambda$ is an $attractor$ then $\Lambda$ is saturated by weak unstable leaves.
		If $\Lambda$ is a $repeller$ then $\Lambda$ is saturated by weak stable leaves.
	\end{property}
	
	From now on we assume the following:
	
	\begin{assumption}
		We assume throughout that the
		Anosov flow $\Phi$ on $\M$ is non transitive and its nonwandering set $\Omega$ consists of two dimensional basic sets only.
	\end{assumption}

	In other words we assume that $\Phi$ has no one dimensional basic set.
	As $\M$ is compact there exits at least  one attracting basic set and one repelling basic set. Suppose $\A$ denotes the union of all attracting basic sets and $\R$ denotes the the union of all repelling basic sets. We will denote the collection of all lifts of $\A$ in $\MM$ by $\widetilde{\A}$. $\widetilde{\A}$ is the the attracting set for $\TPhi_{t}$ defined on $\MM$. The union of all lifts of $\R$ is denoted by $\widetilde{\R}$ similarly. 
	
	\begin{property}\cite{KH95} \label{conv1}
		Suppose $\gamma$ is a flow line not contained in $\A$ or $\R$. Then there exists a flow line in $\A$, say $\alpha$, such 
		that the forward rays of $\gamma$ and $\alpha$ are asymptotic in $\M$. Similarly there exits a flow line $\beta$ in $\R$ such that the backward rays of $\gamma$ and $\beta$ are asymptotic in $\M$.
	\end{property}
	
	\begin{proof} This is classical \cite{KH95}, we explain briefly.
		Given the orbit $\gamma$ it gets closer and closer to the attractor
		$\A$ in future time. Fix $x$ in $\gamma$.
		Every point in the attractor has a local
		product structure, see for example Proposition 6.4.21 of
		\cite{KH95}. Hence for a $t$ sufficiently big $\Phi_t(x)$ is
		$\epsilon$ near the attractor where $\epsilon$ is smaller than
		the size of product boxes of the hyperbolic set $\A$. Hence
		$\Phi_t(x)$ is $\epsilon$ near some point $z$ in $\A$ and there
		is $w$ in $\A$ near $z$ so that $\Phi_t(x)$ is in the stable
		manifold of $w$ because of the local product
		structure in sets of size $\epsilon$. This proves the result.
	\end{proof}
	
	The attractor is saturated by leaves of $\F^{wu}$ and the repeller
	saturated by leaves of $\F^{ws}$.
	In the property above, one can choose the flow line $\alpha$ in the attractor $\A$ to be
	in the boundary of the attractor.
	This means the following: let $x \in \alpha$ and $L$ the $\F^{ws}$ leaf through $x$.
	Let $D$ be a small disk in $L$ with $x$ in the interior. The local 
	flow line of $x$ cuts $D$ into two components $D_1, D_2$ (which are
	also disks). The condition is that one of $D_1$ or $D_2$ does
	not intersect the attractor $\A$.
	Suppose it is $D_1$. The ``$D_1$ side'' of $\alpha$ in $L$ is the
	side so that $\gamma$ is getting closer and closer to $\alpha$.
	
	\section{The foliation $\F$}
	\label{s.foliation}
	
	Throughout the article
	we will fix  a non transitive Anosov flow $\Phi$ as in
	the previous section, that is, $\Phi$ has only two dimensional
	basic sets.
	
	To prove our results 
	we will consider a two dimensional foliation $\F$ in $\M$ such that:
	\begin{itemize}
		\item on the attractor $\A$, \ \ $\F|_{\A}=\F^{wu}|_{\A}$
		\item on the repeller $\R$, \ \ $\F|_{\R}=\F^{ws}|_{\R}$
		\item on $\M\setminus\{\A\cup\R\}$, $\F$ is transversal to both $\F^{ws}$ and $\F^{wu}$.
		\item every leaf $L\in \F$ is subfoliated by the flow lines of $\Phi$, i,e every leaf $L$ is $\Phi_{\RR}$-invariant.  
	\end{itemize}
	We will denote the lift of $\F$ in the universal cover $\MM$ by $\FF$. Leaves of $\FF$ in $\widetilde{\A}$ and $\widetilde{\R}$ look like the leaves in figure \ref{UAS}. Leaves in $\MM\setminus(\widetilde{\A}\cup\widetilde{\R})$ are described in figure \ref{NAR}. It is not immediate from definition why the leaves not contained in $\widetilde{\A}$ and $\widetilde{\R}$ are as described in figure \ref{NAR}, we will prove this later in this
	article.
	
	\begin{theorem}\label{examples} \label{t.existence}
		There are foliations $\F$ with the properties described above.
	\end{theorem}
	
	\begin{proof}
		We start with an Anosov flow as described above.
		For simplicity assume that $M$ is orientable as well. 
		There is a collection of disjoint  tori $\{ T_i \}$ transverse to the
		flow $\Phi$ which separate the basic sets \cite{Sma67, Bru93}.
		We choose $T_i$ to be smooth. 
		The collection of tori is supposed to be minimal with the property that if an
		orbit is not in $\R$ or $\A$ then it intersects one of the $\{ T_i \}$.
		Let $\gamma$ be such an orbit intersecting a specific $T_i$, let
		$x$ be a point in the intersection. Then the forward orbit of $x$
		is asymptotic to a component $A$ of the attractor $\A$ $-$ this
		uses the fact that there are no one dimensional components
		of the non wandering set of $\Phi$. The set
		of such $x$ so that the forward ray of $x$ is asymptotic to $A$
		is open in $T_i$. This holds for any component $A$ of the attractor
		$\A$. Since the union over such components of $\A$ is all of $T_i$
		and $T_i$ is connected, it follows that all orbits in $T_i$ are
		forward asymptotic to a single component $A$ of $\A$.
		
		In a similar way one proves that if $T_1, T_2$ are tori 
		contained in the complement of the union of the attractor and
		repeller, and $T_1, T_2$ intersect a common orbit of $\Phi$,
		then $T_1, T_2$ intersect exactly the same set of orbits
		of $\Phi$. In other words, if $B$ is a component
		of $M - (\A \cup \R)$, then there is a torus $T$ contained
		in $B$, transverse to $\Phi$ so that $B$ is the flow saturation
		of $T$. Hence we can choose a minimal collection $\{ T_i \}$ 
		of tori transverse to $\Phi$ and intersecting
		all orbits in the complement of $\A \cup \R$, and any such
		orbit intersects a unique $T_i$ and only once.
		
		\begin{figure}
			\centering
			\includegraphics[width=0.6\linewidth]{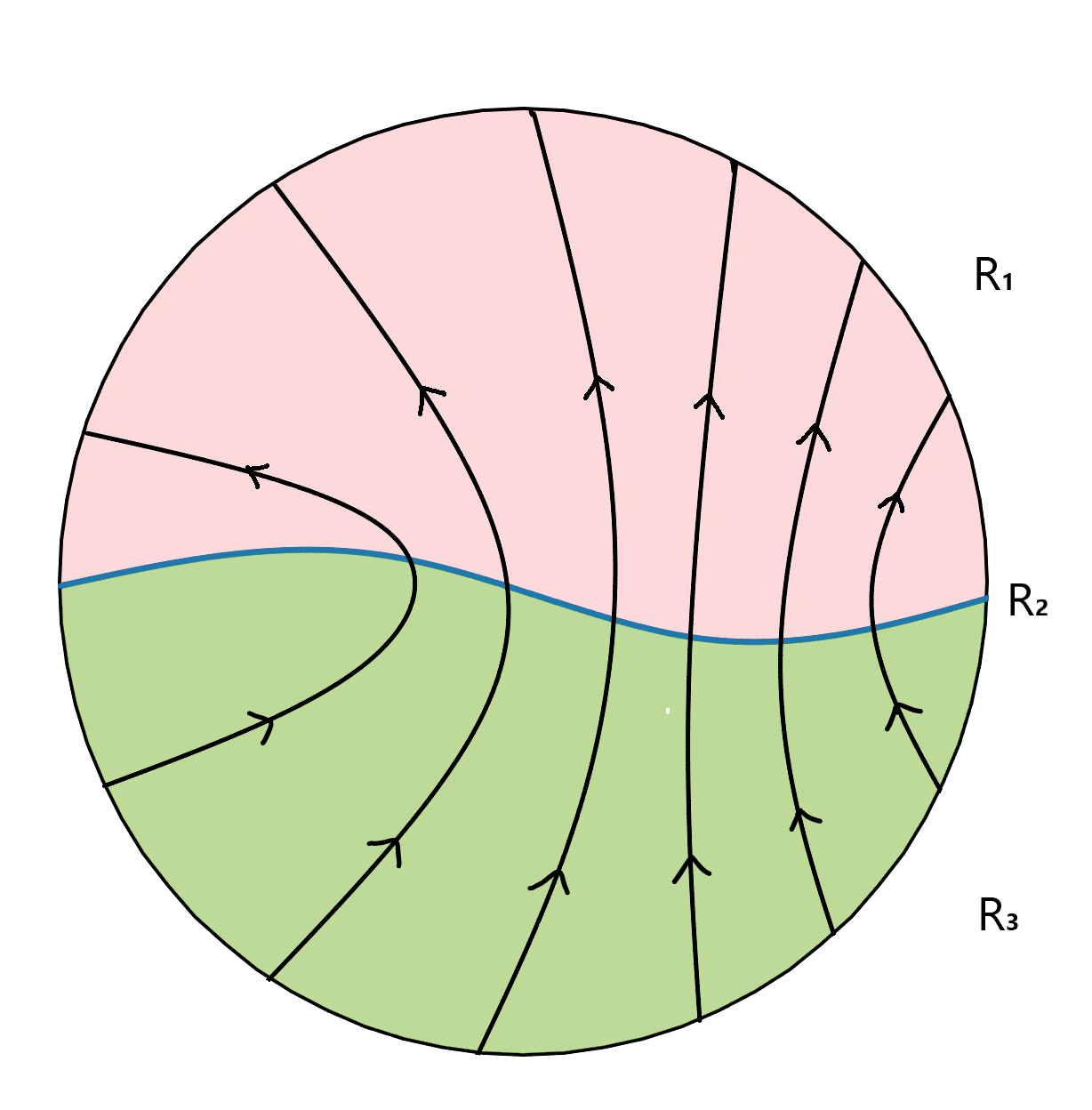}
			\caption{An example of a leaf $L\in\FF$ not contained in $\widetilde{\A}$ or $\widetilde{\R}$. \\In $R_1$ forward rays are asymptotic to $\widetilde{\A}$; in $R_3$ backward rays are asymptotic to $\widetilde{\R}$; $R_2$, the blue line, represents the intersection of $L$ with some lift
				$\widetilde T_i$ of some torus $T_i$. }
			\label{NAR}
		\end{figure}

		\vskip .1in
		\noindent
		{\bf{Construction of $\F$}}
		
		Now we construct the foliation $\F$.
		The foliations $\F^{ws}, \F^{wu}$ 
		are $C^1$ \cite{KH95},  and so are the intersections with each $T_i$.
		On each $T_i$ choose a one dimensional $C^1$ foliation $F_i$
		transverse to both $$\F^{ws} \cap T_i, \ \ \ 
		{\rm and} \ \ \  \F^{wu} \cap T_i.$$
		\noindent
		Saturate $F_i$ by the flow to produce a two dimensional
		foliation in the flow saturation of $T_i$. 
		Let $\F$ be this foliation in the complent of the attractor
		union the repeller.
		Figure \ref{NAR} describes a possible leaf in the lift $\widetilde{\F}$ of $\F$ to $\MM$;
		where $R_2$, the blue line, represents its intersection with some lift of $T_i$. 
		The figure depicts the following
		several properties that we are going to prove
		later and that are essential to the results of this article:
		\ 1) We will show later that leaves of $\FF$ are Gromov
		hyperbolic, \ 2) We will also show that for $L \in \FF$ not in the lift
		of the attractor or repeller then each flow ray in $L$ converges
		to a single point in $S^1(L)$, and that distinct flow rays
		do not forward converge to the same ideal point in $S^1(L)$.
		Similarly for backward flow rays.

		A flow line that does not intersect any $T_i$ has to be either
		in the attractor ($\A$) or the repeller ($\R$). 
		We define $\F$ to be $\F^{wu}$ in the attractor,
		$\F^{ws}$ in the repeller, and the saturation of the $F_i$ 
		everywhere else.
		
		\vskip .1in
		We claim that $\F$ is a foliation. Clearly it is a foliation
		in the complement of the union of 
		the attractor and the repeller, because this is an
		open set and because of the definition of $\F$:
		each component $\CCC$ of $\M \setminus (\A \cup \R)$
		is equal to $\Phi_{\RR}(T_i)$ for some $T_i$ and this is
		homeomorphic to $T_i \times \RR$ with the product topology
		(the topology in $T_i$ is induced from $\M$).
		The foliation $\F$ in $\CCC$ is equivalent to the
		foliation $F_i \times \RR$ in $T_i \times \RR$.
		
		There is a subtle point here. Let $\FF$ be the lift of $\F$ to 
		$\M$. If $\F$ is a foliation, then $\FF$ is a foliation of $\MM$ by
		properly embedded planes. 
		By construction the leaves of $\F$ intersecting the attractor
		are contained in the attractor and similarly for the repeller.
		Therefore the leaves of $\F$ in the complement of $\A \cup \R$
		are entirely contained in the complement of $\A \cup \R$ as well.
		In particular if $L$ is a lift of a
		leaf of $\F$ in the complement of the attractor and repeller, then
		it should be properly embedded in $\MM$. 
		As it turns out this property is not true if one starts with an arbitrary
		foliation $F_i$ in $T_i$. Let us review the construction: we start
		with a foliation $F_i$ in $T_i$ and saturate it by the flow
		to produce a foliation in an open set in $\M$.
		Then consider a lift $L$ of a leaf of this foliation
		to the universal cover. Is $L$ always
		properly embedded in $\MM$? In general this is not true. For example
		start with the Franks-Williams non transitive flow \cite{FW80},
		consider a smooth torus $T$ which separates the attractor
		and repeller and start with say the intersection of the unstable
		foliation of $\Phi$ with $T$, which we call $F$.
		Then for some of the leaves of $F$, it follows that if
		$L$ is a lift of the flow saturation to $\MM$, then $L$ is
		not properly embedded in $\M$. 
		For example  \cite[Fig. 3, page 164]{FW80} depicts the foliations
		induced by the weak stable and unstable foliations in $T$ for
		the Franks-Williams flow. Each has two Reeb components.
		Take $\alpha$ to be a leaf of the unstable foliation which is not
		in the interior of a Reeb component, that is a horizontal
		line in the figure. Lift it to $\widetilde \alpha$ in $\MM$.
		If $C$ is the flow saturation of $\widetilde \alpha$, then 
		$C$ is not properly embedded in $\MM$: there is an orbit $\gamma$ 
		of $\wwp$ which is not in $C$ but is contained in the closure of $C$.
		This orbit $\gamma$ is the lift of a periodic orbit contained in the attractor of the
		Franks-Williams flow.
		
		\vskip .1in
		The reason why our construction of a foliation $\F$ as above
		works is because 
		we start with a foliation $F_i$
		which is transverse to both the stable and unstable foliations
		in $T_i$. 
		We state this as a separate result.
		
		\begin{lemma} \label{lem.emb}
			Let $\ell$ be a leaf of $F_i$ and let $L$ be a 
			lift to $\MM$ of the flow saturation of $\ell$ in 
			$\M$. Then $L$
			is a properly embedded plane in $\MM$.
		\end{lemma}
		
		\begin{proof}{}
			Let $E$ be the flow saturation of $\ell$. Since $\ell$ is smooth
			and the flow is $C^1$ it follows that $E$ is $C^1$. For any
			$x, y$ in $\ell$ if 
			
			$$\Phi_t(x) \ \ = \ \ \Phi_s(y),$$
			
			\noindent
			then $x = y$ and $t = s$, since the component of $\M - (\A \cup \R)$
			containing $\ell$ is homeomorphic to $T \times \RR$ and $\ell$
			is injectively immersed in $T$. Hence $E$ is parametrized as
			$\ell \times \RR$, that is every point $p$ in $E$ can be represented as $(x,t)$ where $x\in \ell$ and $t\in\RR$ .
			
			The Riemannian metric in $\M$ induces a Riemannian metric in $E$
			and a path metric in $E$. We show that with this path metric
			$E$ is complete. In particular this implies that the lift
			$L$ to $\MM$ is a properly embedded plane.
			What we prove is the following:
			
			\begin{claim 0}
				There is $a_0 > 0$ so that any point
				$p = (x, t)$
				in $E$ is the center of a metric disk of radius $a_0$
				in $E$. 
			\end{claim 0}
			
			\begin{proof}
				This is obvious for any point $p$ in $\ell$ or in other words
				if $t = 0$. 
				
				We now prove the claim for $t > 0$ using the unstable foliation.
				The analogous proof shows the result for $t < 0$ using the
				stable foliation.
				The foliation $F_i$ is transverse to both the stable and unstable foliations induced 
				in $T_i$,
				hence uniformly transverse to these
 foliations.
				Given any smoothly embedded curve $\alpha$ in $M$ let 
				$l_u(\alpha)$ be its unstable length: we integrate only the 
				component of the tangent vector in the direction of the 
				unstable bundle. For example if $\alpha$ is contained in a weak
				stable leaf then $l_u(\alpha)$ is zero, while if $\alpha$ is contained
				in a strong unstable leaf then $l_u(\alpha)$ is the same as its
				length under the Riemannian metric of $\M$. In particular if $\alpha$ is a curve not contained in a strong stable or unstable leaf, then original length $l(\alpha)$ is always strictly greater than the unstable length $l_u(\alpha)$. 
				

				By the definition of an Anosov flow, there exist constants $C > 0,\lambda>1$ such that if we flow forward a segment with $t$ amount of time, the new unstable length is at least $C\lambda^t$-times of the original unstable length. 
Hence if we let $a_1 = C$ then for any smooth segment, any flow forward
of that segment has unstable length which is at least $a_1$ times of the original unstable length.
				
				Since $F_i$ is uniformly transverse to 
$\F^{ws}\cap T_i$ by our construction, it follows that
				any point $x$ in $T_i$ is the midpoint of a segment $\beta$ in its leaf
				of $F_i$ of unstable length $2$. For any $t \geq 0$ the unstable
				length of $\Phi_t(\beta)$ is at least $2 a_1$. 
				This constant $a_1$ is defined globally.
				In addition if $v$ is  a non zero
				vector tangent to $\beta$ then $v$ makes a
				definite positive angle with the flow direction. Since flowing
					forward increases the size of unstable vectors more than the
					size of tangent vectors (for $t > t_0 > 0$) it follows that there is a global constant $\theta > 0$
					so that $D\Phi_t v$ also makes an 
angle $> \theta$ with the
					tangent to the flow. 
Consider the infinitesimal arclengths $dt, ds, du$ along the
flow, stable and unstable bundles.
The (non Riemannian) metric 

$$|dt| \ + \ |ds| \ + \ |du|$$

\noindent
is quasicomparable with the Riemannian metric in $M$: there
is $a_2 > 0$ so that the Riemannian length is at least $a_2$ times
the length in this metric.
Consider the following set:

$$A \ \ = \ \ \Phi_{[t-1,t+1]}(\beta)$$

\noindent
for $t \geq 0$.
The segment $\beta$ of $F_i$ is contained in the leaf $E$ of
$\F$. From any point in the boundary of $A$ to $\Phi_t(x)$
along $E$ one has to have at least $a_1$ unstable length,
and flow length of at least $1$.
It follows that there is a global constant
$a_0$ (depending only on $a_1$) so that $A$ contains a disk
in the Riemannian metric, of radius $a_0$ and
centered at $\Phi_t(x)$.

			For $t < 0$ we use the stable foliation and flow backwards
			instead of forwards.

This finishes the proof of the claim.
\end{proof}

				

			The claim shows that $E$ is complete and finishes the proof of
			the lemma.
		\end{proof}

		\noindent
		{\bf {Proof of Theorem \ref{examples}}} $-$ \  
		We consider a foliation $\F$ as constructed
		in the beginning of this section. This object $\F$ is a foliation
		restricted to $\M - (\A \cup \R)$, and this is an open set.
		The leaves of $\F$ in this set lift to properly embedded
		planes in $\MM$ by Lemma \ref{lem.emb}.
		It follows that $\FF$ describes $\MM$ as the disjoin union
		of properly embedded  planes which form a foliation in the complement
		of the lift of the union of the attractor and the repeller.
		
		The only remaining thing to prove is that if a sequence $x_n$ in
		$\M - (\A \cup \R)$ converges to $x$ in $\A \cup \R$ then
		the leaves of $\F$ through $x_n$ converge to the leaf of $\F$ 
		through $x$. Without loss of generality we may assume that $x$ is
		in the attractor.
		
		Let $p_n \in T_i$ so that $x_n$ are in $\Phi_{\RR}(p_n)$.
		There are $t_n \in \RR$ with $x_n = \Phi_{t_n}(p_n)$. Since $x$
		is in the attractor then $t_n$ converges to positive $\infty$.
		The leaf of  $\F$ through $p_n$ is
		the $\Phi$ flow saturation of the leaf of $F_i$ 
		through $p_n$.
		The tangent to this two dimensional
		set through $p_n$ is generated by the Anosov vector field generating
		$\Phi$ and a tangent vector $v$ to $F_i$ at $p_n$. The leaf of $\F$ is $\Phi$-flow
		invariant. Flowing forward, the flow vector remains invariant.
		The vector $v$ is transverse to the weak stable foliation
		and hence it flows more and more to the weak unstable
		direction. So flowing forward these leaves become more and more
		tangent to the $E^0 \oplus E^u$ bundle, and limit to leaves of
		$\F^{wu}$. Since flowing forward limits to the attractor, this shows that the leaves of $\F$ through $x_n$ converge to the leaf
		of $\F$ through $x$. This
		shows that $\F$ defines a foliation.
		This finishes the proof of Theorem \ref{examples}.
	\end{proof}

	We remark that the construction of $\F$ highlights why our methods do not
	work when there are one dimensional basis sets. For simplicity
	suppose that there is a basic set which is a periodic orbit
	$\gamma$.  There is a torus $T$ so that negative saturation
	limits on $\gamma$. If we start with $F$ in $T$ transverse to
	both $\F^{ws} \cap T$ and $\F^{wu} \cap T$ then flowing backwards will make
	it limit to the weak stable leaf of $\gamma$. So the weak
	stable leaf of $\gamma$ is in the collection $\F$ so constructed.
	But there is also a torus $T'$ so that the forward flow
	saturation limits on $\gamma$. The similar argument shows
	that the weak unstable foliation of $\gamma$ also has to be
	in the collection $\F$. Hence the collection $\F$ has sets
	which intersect transversely and cannot be a foliation.

	The structure of the proof of Theorem \ref{main} is as follows:
	We will prove the following properties for such a foliation $\F$:
	\begin{enumerate}

		\item  The flow lines are \textit{leafwise quasigeodesics} in leaves of $\F$, 
		\item Every leaf of $\F$ not contained in $\A$ or $\R$ is a non funnel leaf as in figure \ref{NAR}. 
	\end{enumerate}

	\section{Gromov  hyperbolicity of the Leaves of $\F$}
	
	We will consider a foliation $\F$ as constructed in the previous section.
	
	\par In this section we will show that there exists a Riemannian metric $g$ such that every leaf of the foliation $\F$ is Gromov hyperbolic.
	By \textit{Candel's Uniformization Theorem}, this condition is equivalent to the fact that every holonomy invariant non trivial measure $\mu$ on $\M$ has Euler characteristic $\chi_{\mu}(\M,\F)<0$, which includes the case when there exists no invariant measure. 
	For more details about Euler characteristic see \cite{Can93} or \cite{CC00}. In our context we will
	prove that there is no holonomy invariant transverse measure.
	In fact, under these conditions, 
	Candel proved that there is 
	a metric in $\M$ inducing a smooth metric in the leaves so
	that curvature in each leaf of $\F$ is constant equal to $-1$.
	A precise statements can be found at \cite{Can93},\cite{CC00} or \cite{Cal07}. 
	We call such a metric a \textit{Candel metric}.
	This Candel metric is not smooth in the transverse direction.

	Here is the precise statement on the equivalence of Gromov hyperbolicity of leaves of a foliation and negative Euler charectaristic of a positive invariant measure:

	\begin{proposition}[\cite{Can93}]
		
		Let $(M,F)$ be a compact oriented surface lamination with a Riemannian metric $g$. Then $\chi(M,\mu)<0$ for every positive invariant transverse measure $\mu$ if and only if there is a metric in $M$
		which induces a metric in each leaf of $F$ which makes it 
		into a hyperbolic surface. 
		In particular, this holds true if $M$ has no invariant measure.
	\end{proposition}
	
	To prove that all the leaves of $\F$ are hyperbolic, we will show that there does not exist any invariant measure.  We will argue by contradiction, we assume that there exist a invariant measure $\mu$ and we will attain a contradiction.  

	The $support$ of $\mu$ on $\M$, denoted by $supp(\mu)$, is defined as the collection of all points $x\in\M$ such that if $\tau$ is a one dimensional manifold transverse to $\F$ which contains $x$ in its interior then $\mu(\tau)>0$. The support of a holonomy invariant transverse measure is  a closed set and it is saturated by $\F$, which means $supp(\mu)$ is a union of leaves of $\F$. 
	The orientation hypothesis is not essential as it can be achieved
	by a double cover. The double cover does not change the
	conformal type of any leaf.
	
	\begin{lemma}
		The support of $\mu$ on $\M$ contains at least one leaf from the attractor $\A$ or the repeller $\R$.
	\end{lemma}
	
	\begin{proof} Consider a point $x\in supp(\mu)$ and suppose $L_x$ is the leaf in $\F$ which contains $x$, then $L_x\subset supp(\mu)$ as $supp(\mu)$ is $\F$-saturated. If $x\in \A$, then $L_{x}\subset \A$ and the claim is true. Similarly if $x$ is in $\R$ then its leaf is contained
		in $supp(\mu)$. Finally suppose that
		$x\notin (\A \cup \R)$. Then consider the sequence $\{\Phi_{n}(x)\}$ 
		as $n \rightarrow \infty$.  Let $z$ be an 
		accumulation point of $\{\Phi_{n}(x)\}$. 
		As $supp(\mu)$ is closed, $z$ is in $supp(\mu)$ and hence
		$L_{z}\subset supp(\mu)$. Since $z$ is an accumulation point
		of $\Phi_n(x)$, it implies that $z$ is a non wandering point, hence $z\in\A\cup\R$ and $L_z\subset (\A\cup\R)\cap supp(\mu)$.  
		In fact since $n \rightarrow \infty$, it follows that $z$ is in the 
		attractor, so $L_z \subset \A$.
	\end{proof}
	
	Suppose $L$
	is a leaf in $supp(\mu)$ which is contained in $\A$ (assume in $\A$
	without loss of generality). By 
	\cite[Theorem 6.3]{Pla75}, we know that if $\mu$ is a holonomy invariant 
	transverse measure on a compact manifold foliated by a codimension one foliation $\F$ then any leaf contained in $supp(\mu)$ has polynomial growth. Then the
	leaf $L_z$ in the attractor $\A$ as obtained in
	the previous paragraph has polynomial growth.
	Recall that the leaves of $\F$ are either planes or annuli.
	At the same time, $L_z$ is contained in the attractor and each leaf in the attractor belongs to the weak unstable foliation of the Anosov flow 
	$\Phi$. But weak stable and weak unstable leaves of  Anosov flows have exponential growth, a contradiction. 
	This contradiction shows that 
	each leaf of $\F$ is Gromov hyperbolic.

	\par As each leaf $L\in\FF$ is Gromov hyperbolic with respect 
	to the path metric from the induced Riemannian metric from $\MM$, 
	we can define the circle at infinity or the ideal boundary $S^1(L)$ of each leaf $L$. 
	
	Next we will describe the topology we will use on the spaces 
	
	$$\SC(\MM)=\bigcup\limits_{L\in\FF}\SC(L)\text{ and}$$
	$$\MM \cup \SC(\MM)=\bigcup\limits_{L\in\FF}(L\cup\SC(L))$$
	
	For this we will assume first that $M$ has a Candel metric.
	
	Suppose $\tau$ is an open segment homeomorphic to (0,1) and  transversal to $\FF$. We define the  the following sets $$\mathcal{P}_{\tau}=\bigcup\limits_{y\in\tau}\SC(L_{y})\text{ and }\mathcal{Q}_{\tau}=\bigcup\limits_{y\in\tau}(L_{y}\cup \SC(L_{y}))$$ 
	If $T^{1}(\tau)$ 
	denotes the unit tangent bundle of $\FF$ restricted to $\tau$, then $T^{1}(\tau)$ is naturally homeomorphic to the standard cylinder. The natural identification between $T^{1}(\tau)$ and $\mathcal{P}_{\tau}$ induces the topology on $\mathcal{P}_{\tau}$ homeomorphic to the standard annulus. 
	In \cite{Fen02} 
	it is  proved that this is independent of the
	transversal $\tau$ that is chosen intersecting the same sets of
	leaves of $\FF$.
	This is because the metrics induced in $S^1(L)$ from the visual
	metric in any point are H\"{o}lder equivalent.
	
	Similarly $\mathcal{Q}_{\tau}$ has a natural topology homeomorphic to the standard solid cylinder. 
	
	The collection of all $\mathcal{P}_{\tau}$'s over a $\pi_1(M)$-invariant discrete collection of transversals  defines a topology on $\SC(\MM)$. Similarly the collection of $\mathcal{Q}_{\tau}$s over the same collection
	of transversals defines a topology on $\MM \cup \SC(\MM)$. 
	Deck transformations act by homeomorphisms on both sets.
	For more details see \cite{Fen02}, \cite{Cal00} or \cite{Cal07}.
	
	After the fact it is easy to see that the topologies described are
	independent of the specific metric in $M$ chosen and also work for any 
	Riemannian metric in $M$.

	\section{Properties of flow lines}
	
	\par This section describes the behavior of forward rays of 
	flow lines, in particular their asymptotic behavior towards the
	the boundary at infinity $\bigcup\limits_{L\in\FF}\SC(L)$. 
	In particular we will prove that the rays are quasigeodesics
	in their respective leaves of $\FF$. Notice that this is
	definitely much weaker than saying that full flow lines are
	quasigeodesics in their respective leaves. We will also show
	that in some leaves, the forward ideal points are pairwise
	distinct and the negative ideal points are also pairwise
	distinct. In particular even if the flow foliation is a
	leafwise quasigeodesic subfoliation of $\F$ it will not
	have the funnel property.

	We now introduce a family of sets in $\MM$ which will be 
	extremely useful for us:
	
	\vskip .1in
	\noindent
	{\bf {The sets $\U$}}
	
	\par Fix a point $x\in\widetilde{\A}\subset\MM$ and the forward ray from
	$x$ which is $\gamma^{+}_{x}=\TPhi_{[0,\infty)}(x)$ starting at $x$, let $L_x\subset\widetilde{\A}$ be the leaf containing $\gamma^{+}_{x}$. 
	Recall that in the attractor $\A$ the foliation
	$\F$ is equal to $\F^{wu}$, hence transverse to $\F^{ws}$.
	Therefore the foliations $\FF$ and $\FF^{ws}$ are transversal to each other near $\widetilde{\A}$. 
	
	Let $U$ be a compact rectangle transverse to the flow 
	and with
	$x$ in the interior of $U$. We assume that $U$ is contained
	in foliation boxes of all foliations, 
	that $U$ is made up of a union of 
	stable segments, every one of which intersects the
	local strong unstable segment through $x$.
	Consider the set 
	$$\U \ = \ \TPhi_{[0,\infty)}(U)$$
	
	\noindent
	The set $\U$ is a neighborhood of the forward ray $\TPhi_{[0,\infty)}(x)$. We can assume that $\U$ is homeomorphic to $[-1,1]\times[-1,1]\times[0,\infty)$ with $x=(0,0,0)$ and we can define coordinates on $\U$ such that 
	\begin{itemize}
		\item $U$ is identified with $[-1,1]\times[-1,1]\times\{0\}$ and points on $U$ are represented as $(r,s,0)$ for $r,s\in [-1,1]$. In particular, $x=(0,0,0)$. 
		\item for a point $y=(r,s,0)$, $\TPhi_t(y)$ has coordinates
		$(r,s,t)$, that is, the  ray $\{(r,s,t)|t\in[0,\infty)\}$ represents the ray $\Phi_{[0,\infty)}(y)$.
		\item for a point  $y'=(r',s,t')\in \U$, ${P_{y'}}$ denotes the horizontal infinite strip  $$P_{y'}=\{(r,s,t)|r\in[-1,1],t\in[0,\infty)\}$$ The infinite strip  $P_{y'}$ is contained in the leaf $L_{y'}\in\FF$ which contains $y'$.   
		\item for a point $y'=(r,s',t')\in \U$, ${Q_{y'}}$ denotes the vertical infinite strip  $$Q_{y'}=\{(r,s,t)|s\in[-1,1],t\in[0,\infty)\}.$$ The infinite strip  $Q_{y'}$ is contained in the leaf $E_{y'}\in\FF^{ws}$ which contains $y'$. 
	\end{itemize}
	As $x=(0,0,0)\in\widetilde{\A}$ 
	the leaf of $\FF$ through $x$ is actually the weak unstable leaf of
	$\widetilde \Phi$ through $x$, hence 
	$P_{x}$ is contained in the $\FF^{wu}$ leaf through $x$.
	
	\vskip .1in
	The sets $\U$ will be used throughout this section.
	
	We can define a projection map $\Pi: \U \rightarrow P_x$  by
	the formula $\Pi(y)=S_{y} \cap P_{x}$ where $S_{y}$ is the one dimensional leaf of the strong stable foliation $\FF^{s}$ containing $y$. This is possible because
	one can do that in the original rectangle $U$ as it is a union
	of strong stable segments, and then $\U$ is the flow forward saturation
	of $U$ and the maps $\widetilde \Phi_t$ preserve the strong
	stable foliation in $\widetilde \M$.
	These projection maps are well defined and continuous because of the foliation structures on $\U$.

	For any $y\in\U$, the rays $\widetilde{\Phi}_{[0,\infty)}(y)$ and $\widetilde{\Phi}_{[0,\infty)}(\Pi(y))$ are asymptotic as they lie on the same weak stable leaf. We assume that lengths of all the line segments $\{(r,s,0)|s\in[-1,1]\}$ are less than a fixed $\epsilon>0$ in $\MM$. 
	
	The line segment $\lambda=\{(0,s,0)|s\in[-1,1]\}$ is transversal to $\FF$. Consider the open sets $\mathcal{V}=\bigcup\limits_{x'\in\lambda}\SC(L_{x'})$ and $\mathcal{W}=\bigcup\limits_{x'\in\lambda}(L_{x'}\cup\SC(L_{x'}))$.
	
	\begin{definition} Let $\gamma$ be a flow line of $\widetilde \Phi$
		contained in a leaf $L$ of $\FF$. Given $x$ in $\gamma$ if the
		forward ray converges to a single point of $S^1(L)$ we let this
		be $\eta^+(x)$. Similarly define $\eta^-(x)$.
		In addition given a point $a$ in $\MM$ let $\gamma_a$ be
		the flow line of $\wwp$ containing $a$.
	\end{definition}

	\begin{lemma}
		For any $w\in\MM$, the forward and the backward rays of the flow line $\gamma_{w}= \widetilde \Phi_{\RR}(w)$ are quasigeodesics on the leaf $L_w$ in $\FF$ which contains the flow line.
	\end{lemma}
	
	\begin{proof}
		\begin{figure}
			\centering
			\includegraphics[width=0.65\linewidth]{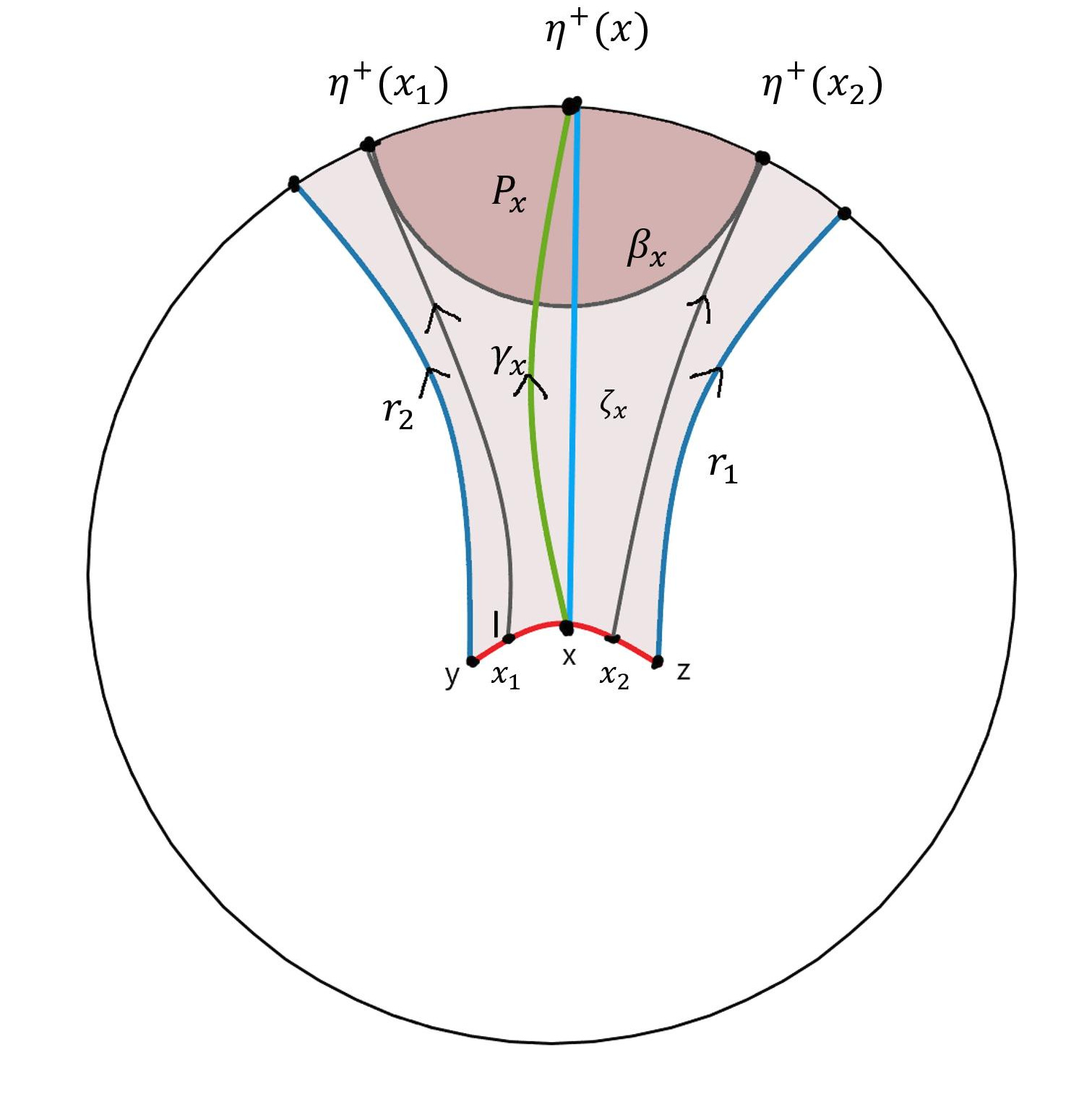}
			\caption{the region $A_x$ in $L_x$ and the half-space $P_x$}
			\label{fig1}
		\end{figure}
		
		To prove this Lemma we assume a Candel metric in
		$\M$ so that
		leaves of $\F$ are hyperbolic surfaces. This metric is not Riemannian,
		but the result is independent of the metric.
		
		By lemma \ref{lqg} every forward or backward ray in a leaf  in $\FF^{wu}$ or $\FF^{ws}$ is quasigeodesic in its respective leaf. 
		In particular every flow line is a quasigeodesic in the respective
		leaf of $\F$ if contained in the attractor or repeller.
		
		So we may assume that the ray is in a leaf not in the attractor or
		repeller. Property \ref{conv1} shows that the ray is asymptotic
		with a ray in the attractor. By taking a subray we may assume that
		the ray is in the weak stable leaf of a point $x$ in the attractor and 
		the initial point $w$ of the ray is very near $x$. Hence we may assume that the initial
		point is contained in a local cross section $U$ 
		to $\widetilde \Phi$ centered at $x$ as described above.
		Let $L_x$ be the leaf of $\FF$ containing $x$, and similarly
		define $L_w$.
		
		Recall that $L_x$ is also  the weak unstable leaf of $\widetilde \Phi$
		containing $x$.
		
		Therefore it is
		sufficient to show that every forward ray in the set $\U$ described
		above  is quasigeodesic in its respective leaf of $\FF$.
		In the $\FF$ leaf through $x$ we consider the following curve.
		Let $I$ be the compact unstable segment $U \cap L_x$ which has
		endpoints $z, y$. Let $r_1, r_2$ be the forward rays of 
		$\wwp$ through
		$z_,y$. Then $c := r_1 \cup I \cup r_2$ is a bi-infinite curve as shown in figure \ref{fig1}.
		The two rays $r_1, r_2$ are quasigeodesics in $L_x$ and
		they converge to distinct ideal points in $S^1(L)$.
		
		The curve $c$ bounds a region $A$ in $L_x$ (as in figure \ref{fig1}) which is exactly
		$\widetilde \Phi_{[0,\infty)}(I)$. This is contained in $\U$.
		This region contains a half plane in $L_x$.
		
		Recall that we are considering $w$ a point in 
		$U \cap \FF^s(x)$, where $\FF^s(x)$ is the strong stable leaf of $x$. Let $J$ be the intersection
		of $L_w \cap U$, where $L_w$ is the leaf of $\FF$ through $w$.
		Then $B := \widetilde \Phi_{[0,\infty)}(J)$ is contained in 
		$L_w$ and contained in $\U$.
		In addition since every point in $J$ is in the strong stable leaf
		of a point in $I$, it follows that every flow ray in $B$ is
		asymptotic to a flow ray in $A$. In fact as points leave 
		compact sets in $B$ they  become closer and 
		closer to $A$.

		The induced metrics on the leaves $\FF$ vary continuously and the ray
		$r_x = \widetilde  \Phi_{[0,\infty)}(x)$ is quasigeodesic in its leaf
		and asymptotic to the ray 
		$r_w = \widetilde  \Phi_{[0,\infty})(w)$, it follows that the other ray is 
		also a quasigeodesic in its $\FF$ leaf.

		Since this is a very subtle point we provide specific details.
		In the leaf $L_x$ choose two points $x_1, x_2$ in $I$ with $x$ in 
		between them so that the geodesic $\beta_x$ in $L_x$ with ideal
		points $\eta^+(x_1), \eta^+(x_2)$ is contained in the interior
		of $A$. This is possible since the flow lines in $L_x$ are
		uniform quasigeodesics and they spread out in the forward direction.
		We stress that in general it is not possible to choose $x_1, x_2$ as
		the endpoints of $I$ as the flow lines are only quasigeodesics and not
		geodesics in $L_x$.
		Let $P_x$ be the half plane of $L_x$ bounded by $\beta_x$ and containing
		a forward ray from $x$.
		We also may assume that every point in $P_x$ is $\epsilon_1$ 
		close to $L_w$ with $\epsilon_1$ 
		very close to zero.
		Then $\beta_x$ is $\epsilon_1$ close to a curve $\beta'$ in $L_w$
		which has geodesic curvature very close to zero. 
		To obtain this property of $\beta'$ with small geodesic curvature
		in $L_w$ was one of the reasons to choose a Candel metric with 
		hyperbolic leaves varying continuously.
		In particular
		this curve $\beta'$ is very close in $L_w$ to an actual geodesic 
		in $L_w$, and this geodesic is  denoted by $\beta_w$. Let $P_w$ be the half plane
		of $L_w$ which is very close to $P_x$ as shown in figure \ref{fig2}.
		\begin{figure}
			\centering
			\includegraphics[width=1\linewidth]{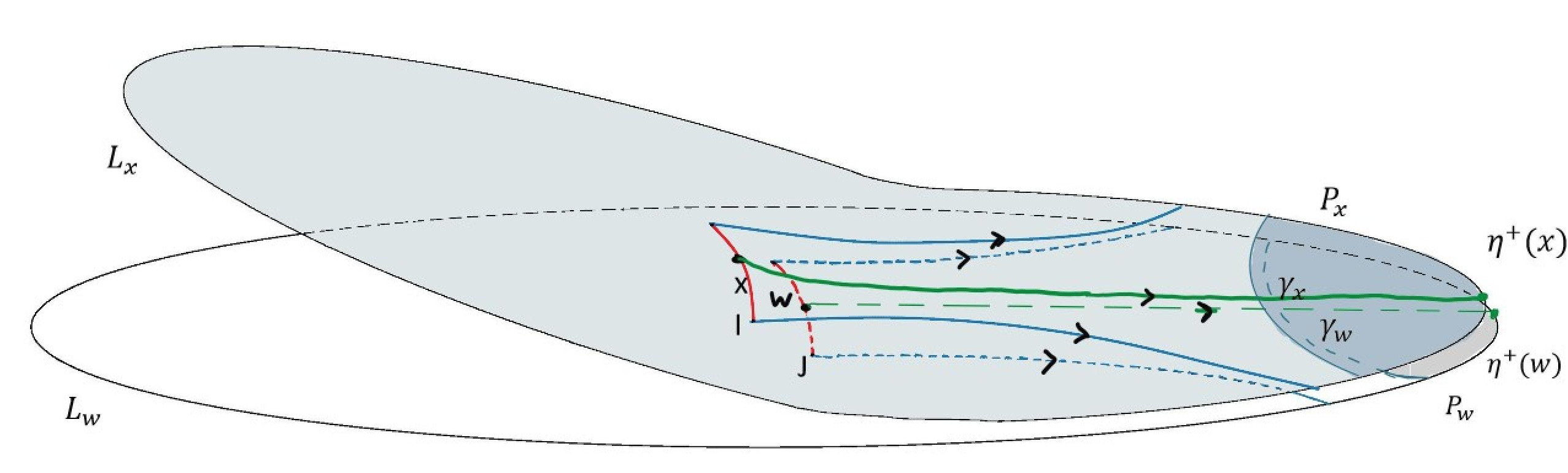}
			\caption{$P_x$ in $L_x$ is asymptotic to $P_w$ in $L_w$}
			\label{fig2}
		\end{figure}

		Now we prove that the ray $r_w$ is quasigeodesic in $L_w$.
		The ray $r_x$ in $L_x$ satisfies the quasi-isometric
		property for some $(K,s)$. The ray $r_w$ is asymptotic to $r_x$ so
		length along $r_w$ is extremely well approximated by length along
		$r_x$ when moving forward. But distance in $L_w$ between
		points in $r_w$ is also boundedly approximated
		by the corresponding distance in $L_x$.
		This is because if one gets a length minimizing path
		in $L_w$ connecting the endpoints of a segment in $r_w$, 
		then this segment is contained in the half plane $P_w$ as above
		if the points are far enough in $L_w$ from $w$.
		This is why we constructed $P_w$. Since $P_w$ is $\epsilon_1$ 
		close to $P_x$ there is a corresponding segment in $P_x$
		whose length is multiplicatively increased by at most a factor
		of $1 + \epsilon_2$, where $\epsilon_2$ is very small.
		But the approximating segment in a flow line in $A$ is a
		quasigeodesic in $L_x$ with constants $(K,s)$, hence the length
		of the approximating path is bounded below, and so
		is the length of the length minimizing original path in $P_w$.
		
		This proves that $r_w$ is a quasigeodesic in $L_w$.
		
		If we reverse the flow every backward ray becomes forward ray, hence leafwise quasigeodesic.  
	\end{proof}
	
	By compactness and continuity there is $K_0,s_0$ so that given
	any ray in a flow line, there is a subray of it that is
	a $(K_0,s_0)$ quasigeodesic in its leaf of $\FF$.
	As the flow rays are quasigeodesics they define unique points on the ideal boundaries.
	
	In the next proposition we consider the sets $P_y$ contained in $\U$.
	
	\begin{proposition}\label{nf}
		Suppose $a,b\in P_y$ but $\gamma_{a}\neq \gamma_{b}$, then $\eta^{+}(a)\neq\eta^{+}(b)$ in $S^1(L_y)$. 
	\end{proposition}
	\begin{proof} 
		By the previous lemma we already know that all rays are quasigeodesics
		in their respective leaves.
		We do the proof by contradiction and assume that $\eta^+(a) = \eta^+(b)$.
		Since the rays $\wwp_{\zein}(a), \wwp_{\zein}(b)$ are quasigeodesics
		in $L_y$ and by assumption they have the same ideal point in $S^1(L_y)$, the
		following happens:
		there is $d_0 > 0$ and points $p_i, q_i$ in 
		$\wwp_{\zein}(a), \wwp_{\zein}(b)$ respectively, escaping in
		the rays and so that $d_{L_y}(p_i,q_i) < d_0$.
		Consider the points $\Pi(a)$ and $\Pi(b)$ on $P_x$. 
		Since 
		$$\wwp_{\zein}(a), \  \wwp_{\zein}(\Pi(a))$$ 
		
		\noindent
		are asymptotic
		in the weak stable leaf of $\wwp$ in $\MM$,
		there are $p'_i$ in $\wwp_{\zein}(\Pi(a))$ with
		$d(p_i,p'_i) \rightarrow 0$.
		Here $d$ is ambient distance in $\MM$.
		Similarly there are $q'_i$ in $\wwp_{\zein}(\Pi(b))$ with
		$d(q_i,q'_i) \rightarrow 0$. 
		By the local product structure of the foliation $\F$ it follows
		that $d_{L_x}(p'_i,q'_i) < d_0 + 1$ for $i$ sufficiently big.
		Therefore the rays $\wwp_{\zein}(\Pi(a))$, \ $\wwp_{\zein}(\Pi(b))$
		converge to the same ideal point in $S^1(L_x)$.
		Here we are using that $L_x$ is also a weak unstable leaf of $\wwp$
		and the flow lines are quasigeodesics in the weak unstable leaves
		of $\wwp$.
		But the flow lines $\wwp_{\RR}(\Pi(a)), \wwp_{\RR}(\Pi(b))$ are
		distinct  flow lines in $L_x$. Again by the description \ref{2.4} of
		ideal points of flow lines in weak unstable leaves, the forward
		limit points are distinct, that is,
		
		$$\eta^+(\Pi(a)) \ \neq \ \eta^+(\Pi(b)) \ \ \ 
		{\rm in} \ \ \ S^1(L_x) $$
		
		\noindent
		This is a contradiction and shows that $\eta^+(a) \neq \eta^+(b)$
		in $S^1(L_y)$.
	\end{proof}
	
	\begin{lemma} \label{hausdorff}
		In each leaf $L$ of $\FF$ the leaf space of the flow foliation is Hausdorff
		and homeomorphic to the reals $\RR$.
	\end{lemma}
	
	\begin{proof}
		In the lifts of leaves in the attractor and repeller this is obvious since
		the foliation by flow lines satisfies this property 
		in weak stable and weak unstable
		leaves of Anosov flows \cite{Fen94}.
		Any other leaf $L$ of $\FF$ is the lift of a leaf of $\F$ which
		intersects a torus $T$ 
		from the collection of tori $\{T_i\}$ which separates $\A$ and $\R$. Hence $L$ intersects a lift $\wT$ of $T$ in
		a curve $\beta$. The flow saturation of $\beta$ is exactly
		$L$, since every flow line in $M$ is either in the attractor
		or repeller; or
		intersects a torus in $\{T_i\}$.
		The curve $\beta$ is transverse to the weak stable and weak 
		unstable foliations, hence intersects a flow line exactly once.
		Hence $\beta$ parametrizes the set of center leaves in
		$L$. This proves the result.
	\end{proof}
	
	For each $L$ of $\FF$  the map $\eta^+$ induces a map from the flow foliation
	leaf space in $L$ (which is $\cong \RR$) to $S^1(L)$. Since
	flow lines are disjoint in their leaves, 
	this map is weakly monotone.

	\begin{corollary} \label{c.different}
		For all $y\in \U$, $\eta^{+}(y)\neq\eta^{-}(y)$ in $S^1(L_y)$.
	\end{corollary}
	\begin{proof} If $\eta^{+}(y)=\eta^{-}(y)$ then the flow line $\gamma_{y}$ bounds a disk $\mathcal{D}$ on $L_y\cup\SC(L_y)$ such that the closure
		of $\mathcal{D}$ in $L \cup S^1(L)$ intersects
		$S^1(L)$ only in $\eta^{+}(y)=\eta^{-}(y)$. For any $z$ in interior of $\mathcal{D}$, the flow line $\gamma_{z}$ is contained in $\mathcal{D}$, hence $\eta^{+}(z)=\eta^{-}(z)=\eta^{+}(y)=\eta^{-}(y)$. This contradicts 
		the proposition \ref{nf}, because if  $z$, $y\in\U$ and $\gamma_z \neq \gamma_y$,
		then $\eta^+(z) \neq \eta^+(y)$.
	\end{proof}
	
	We now extend the map $\eta^+$ to a map from $\MM$ to 
	$\SC(\MM)$.
	For each $x$ in $\MM$ then $\eta^+(x)$ is in $S^1(L_x) \subset
	\SC(\MM)$.

	\begin{proposition}\label{Cont}
		$\eta^{+}$ and $\eta^{-}$ are continuous on $\MM$. 
	\end{proposition}
	\begin{proof}
		In this proof we again use a Candel metric in $M$.

		Suppose $x_i\rightarrow x_0$ in $\MM$. We will show that $\eta^{+}(x_i)\rightarrow \eta^{+}(x_0)$ in $\SC(\MM)$.
		There are two different cases depending on whether $x_0\in\widetilde{\R}$ or $x_0\notin\widetilde{\R}$.
		
		We first prove the result for $x_0\notin\widetilde{\R}$. As $x_0\notin\widetilde{\R}$ the forward ray starting at $x_0$ is asymptotic to a forward flow ray in $\widetilde{\A}$. Therefore it is enough to assume that $\{x_i\}$ and $x_0$ belong to a neighborhood $\U$ as constructed 
		above,  since this is true for every ray asymptoptic to  $\widetilde{\A}$.
		
		For $z$ in $\M$ let $L_z$ be the leaf of $\FF$ containing $z$.
		
		For for $i\in\NN\cup \{0\}$, let
		$\gamma^{+}_i$ denote the forward flow
		ray starting from $x_i$ and let $\zeta_{i}$ 
		denote the geodesic ray on $L_{x_i}$ 
		starting at $x_i$ and with ideal point $\eta^{+}(x_i)$ in $S^1(L_{x_i})$.
		Each $\zeta_i$ defines the ideal point $\eta^{+}(x_i)$ on $S^1(L_{x_i})$, therefore it is enough to show that any convergent subsequence of $(\zeta_{i})$ converges to $\zeta_0$ in the compact open topology. 
		Since $x_i$ are in a compact subset of $\MM$, 
		existence of of convergent subsequences of $\{\zeta_i\}$ is assured.
		
		Suppose that a subsequence
		$(\zeta_{i(k)})$ converges to $\zeta'$. We have to prove that
		$\zeta'=\zeta_0$. 
		We assume that the neighborhood $\U$ constructed above has a
		point $x \in \widetilde \A$ as in the construction 
		of $\U$. Then all flow rays in $L_x \cap \U$
		are $(K,s)$-quasigeodesics in $L_x$ for some fixed $K, s$.
		Since all flow rays in $\U$ are forward asymptotic to flow rays
		in $L_x$ there are $K',s'$ so that all flow rays in $\U$ are
		$(K',s')$-quasigeodesics in their respective $\FF$ leaves.
		It follows that there exists a uniform $d > 0$ such that 
		
		$$\gamma^+_{i(k)} \ \subset \ \mathcal{N}_d(\zeta_{i(k)}) \ \ \ {\rm and}
		\ \ \  \gamma^{+}_{0} \ \subset \ \mathcal{N}_d(\zeta_0),$$
		
		\noindent
		where $\mathcal{N}_d$ denotes the neighborhood of radius $d$ in
		the respective leaf of $\FF$.
		For any $d_1 > 0$, the segment of length $d_1$ on $\gamma^{+}_{i(k)}$ starting at $x_{i(k)}$ is within $d$-distance from $\zeta_{i(k)}$.  Therefore in the limit the segment of $\gamma^+_0$ of length $d_1$ starting from $x_0$ is contained in $d$ distant neighborhood from $\zeta'$ in the respective leaf.
		This is true for all $d_1$, so $\zeta'$ is at Hausdorff distance $d$ from $\gamma^{+}_0$ on $L_{x_0}$. But $\gamma^{+}_0$ is also at bounded distance from $\zeta_0$ on $L_{x_0}$, therefore $\zeta'$ and $\zeta_0$ are at a finite Hausdorff distance from each other on $L_{x_0}$. Hence $\zeta'=\zeta_0$, because
		they have the same starting point.
		As this is true for all convergent subsequences of 
		$(\zeta_i)$, we get our result for $x_0$ not in $\widetilde{\R}$. 
		
		Before dealing with the remaining case let us note the following:
		
		\begin{observation}\label{obs}
			By the construction of $\U$ starting with $x$ in $\widetilde{\A}$,
			and continuity of $\eta^{+}$ near $\widetilde{\A}$ we observe that the set $\U\cup\{\eta^{+}(z)|z\in\U\}$ is homeomorphic to $[0,1]\times[0,1]\times[0,1]$ inside
			$\mathcal{W}=\bigcup\limits_{y\in\lambda}(L_y\cup S^1(L_y))$, which is homeomorphic to a compact solid cylinder $[0,1]\times\{\ the\  unit\  disc\   \mathbb{D}\}$. 
			
			\noindent
			The set $\U\cup\{\eta^{+}(z)|z\in\U\}$ above 
			is saturated by forward flow lines and all the ideal points contained in this neighborhood are defined by forward flow rays. This is true for example
			for any forward ideal point $p$ on any $S^1(L)$ where $L$ is a leaf
			of $\FF$ in $\widetilde{\A}$. 
			Here $\lambda$ is a transversal to $\FF$ intersecting exactly the leaves
			of $\FF$ which intersect $\U$.
			This in particular implies that for any $y$ in $\lambda$ and
			ideal point $v$ in $S^1(L_y)$ which is a forward ideal point
			of $\wwp_{\RR}(z)$ with $z$ in $L_y \cap \U$ and in the interior
			of $\U$ then $\eta^+(z)$ is an interior point of the interval $I_y$
			of $S^1(L_y)$ associated to all forward ideal points of flow
			lines in $L_y \cap \U$. In particular we stress the
			important fact that any ideal point in
			this interval  $I_y$ in $S^1(L_y)$ is an ideal point of a forward flow
			ray, but it is not an ideal point of a backwards
			flow ray.
			
			In an analogous way the corresponding
			property is true for any backward ideal point $q$ on any $S^1(E)$ where $E\subset \widetilde{\R}$, but the difference is that the neighborhood around a backward flow ray defining $q$ is saturated by backward flowrays and all the ideal points in that neighborhood are defined by backward flowrays and
			no such ideal point is an ideal point of a forward flow ray.  
		\end{observation}
		
		To continue the proof of Proposition \ref{Cont} we next 
		assume that $x_0\in \widetilde{\R}$. Suppose that a subsequence
		$(\eta^{+}(x_{i(k)}))$ converges to $q$ where $q$ is not $\eta^+(x_0)$.
		As $x_0$ is in $\widetilde{\R}$, then $L_{x_0}$ is a leaf of the weak stable foliation $\FF^{ws}$. Hence by property \ref{2.4} on $L_{x_0}$ all the forward flow rays converge to a single ideal point in $S^1(L_{x_0})$ and all the other ideal points in $S^1(L_{x_0})$
		are ideal points of backward flow rays.
		As $q\neq\eta^{+}(x_0)$, $q$ is defined by a backward ray, that
		is $q = \eta^-(z)$ for some $z$ in $L_{x_0}$.
		By Observation \ref{obs} starting with $z$ in $\widetilde{\R}$
		(notice that $z$ is in the repeller, not the attractor),
		there exits a neighborhood $\V$
		saturated by backward flow rays around $z$ in 
		$\bigcup\limits_{y\in\lambda'}(L_y\cup S^1(L_y))$ for some transversal $\lambda'$. 
		By Observation \ref{obs} all limit points 
		are backward ideal points in $\V$ and no limit point is a forward ideal point.
		This contradicts the fact that the forward rays $\gamma^{+}_{i(k)}$ have
		ideal points in these intervals of ideal points for $k$ big enough
		by construction. 
		This contradiction shows that a subsequence 
		$(\eta^+(x_{i(k)}))$ 
		converging to $q \neq \eta^+(x_0)$ is not possible, hence $q=\eta^+(x_0)$.
		
		Hence $\eta^{+}$ is continuous on $\MM$. If we consider the flow $\Psi_t=\Phi_{-t}$ then backward ideal points of $\Phi_t$ are forward ideal points of $\Psi_{-t}$ and the continuity of $\eta^{-}$ follows. This completes the proof of Proposition \ref{Cont}.
	\end{proof}
	
	In the next lemma we combine all the above results and describe a key property that will be used to show that all the flow lines are quasigeodesic on their respective leaves of $\FF$. 
	
	Again we use a Candel metric.
	As before, given  $x$ in $\MM$, let $\gamma_x$ be the $\wwp$ flow line
	through it. In addition if $L_x$ is the leaf of $\FF$ containing $x$, 
	let $g_x$ be the geodesic in $L$ with ideal points $\eta^+(x), \eta^-(x)$.
	Notice that we already proved that $\eta^+(x), \eta^-(x)$ exist
	and are distinct from each other.
	This follows from Corollary \ref{c.different}.
	
	\begin{lemma}\label{geodesic nbd}
		There exists $d>0$ such that for all $x\in\MM$ we have that
		$$\gamma_x \ \ \subset \ \ \mathcal{N}_{d}(g_x),$$
		
		\noindent
		where $g_x$ is the geodesic on $L_x$ connecting $\eta^{+}(x)$ and $\eta^{-}(x)$ and $\mathcal{N}_{d}(g_x)$ is the $d$-neighborhood of $g_x$ on $L_x$.  
	\end{lemma}
	
	\begin{proof}Suppose that there does not exists any such $d$. Then there exists a sequence $(x_i)$ in $\MM$ with
		$x_i$ in leaves $L_{x_i}$ of $\FF$ such that $d_{L_{x_i}}(x_i, g_{x_i})>i$. Up to deck transformations there exists a convergent subsequence of $(x_i)$ which we assume is the original
		sequence, and we assume $x_i\rightarrow x$. By lemma $\ref{Cont}$ we know that 
		$$\eta^{+}(x_i)\ \rightarrow \ \eta^{+}(x) \ \ \ {\rm and} \ \ \ 
		\eta^{-}(x_i) \ \rightarrow \ \eta^{-}(x).$$ 
		\noindent
		Since $x_i$ converges to $x$ we assume that all $x_i$ are in leaves
		of $\FF$ which intersect a fixed transveral $\lambda$ to $\FF$.
		
		Since $\eta^+(x_i)$ converges to $\eta^+(x)$, \ \ $\eta^-(x_i)$
		converges to $\eta^-(x)$ and $\eta^+(x) \neq \eta^-(x)$ it
		follows that 
		$\{g_{x_i}\}$ converges to $g_x$. 
		This uses that the  
		topology defined on $\bigcup\limits_{y\in\lambda}(S^1(L_y))$ is
		given by the trivialization of the unit tangent bundle to $\FF$
		along $\lambda$.
		By convergence we mean convergence in the compact open topology.
		But this contradicts that $d_{L_{x_i}}(x_i,g_{x_i})$ converges
		to infinity, since $d_{L_x}(x,g_x)$ is finite and the sequence
		converges to it. 
		This finishes the proof.
	\end{proof}

	\section{Flow lines Are Leafwise Quasigeodesic}

	We first prove a weak quasigeodesic property of the flow lines on the leaves of $\FF$ containing them. 
	
	\begin{proposition}\label{wqg}
		For all $b>0$ there exists a constant $c_b>0$ depending on $b$ such that if $\gamma$ is a flow segment connecting $x$ and $y$ with $length(\gamma)>c_b$ then $d_{L_x}(x,y)>b$, where $L_x$ is the leaf of $\FF$ which contains $x$.
	\end{proposition}
	
	\begin{proof} Fix $b>0$. We do the proof by contradiction. Suppose the statement is not true for some $b>0$. Then for all $i\in\NN$ there exists two points $x_i$ and $y_i$ in leaves $L_i$ of $\FF$, with 
		$x_i, y_i$ in the same flow line defining a flow line 
		segment $\gamma_i$ satisfying
		$length(\gamma_i)>i$ but $d_{L_i}(x_i, y_i)<b$.
		Up to deck transformations and a subsequence, we assume that 
		$(x_i)$ is convergent and $x_{i}\rightarrow x_0$. Since
		$d_{L_i}(x_i,y_i) < b$ we can similarly
		assume that $( y_i )$ is convergent and let
		$y_i\rightarrow y_0$. 
		
		\begin{claim 1}
			\textit{$x_0$ and $y_0$ are on the same leaf $L_0$ of $\FF$. }
		\end{claim 1}
		\noindent
		\begin{proof}
			If we consider a compact ball $B_{x_0}$ on $L_0$ containing
			$x_0$  and a product neighborhood of $N(B_{x_0})$ of $\FF$, then for all large $i$, $L_i$ intersects $N(B_{x_0})$ 
			and $x_i\in L_i\cap N(B_{x_0})$. If we consider $B_{x_0}$ sufficiently large, the assumption $d_{L_i}(x_i,y_i)<b$ for all $i$ 
			forces that $y_i$ has to be contained in $N(B_{x_0})$. Hence by the product structure on $N(B_{x_0})$,  $y_0$ also has to be on $L_0$ as $y_i\rightarrow y_0$. 
		\end{proof}
		\begin{claim 2}
			$x_0$ and $y_0$ cannot be on the 
			same flow line in $L_0$.
		\end{claim 2}
		\noindent
		\begin{proof} If not, then there exists a flow line segment $\gamma$ connecting $x_0$ and $y_0$ and consider a compact neighborhood $\mathcal{N}$ around $\gamma$ which has a product structure with respect to the flow lines. 
			This is the crucial fact. As $x_i\rightarrow x_0$ and $y_i\rightarrow y_0$,  the flow segments $\gamma_i$ are contained in $\mathcal{N}$ for all large $i$. By continuity of length of flow lines,
			$length(\gamma_i)\rightarrow length(\gamma)$. But that is not possible as $length(\gamma_i)\rightarrow \infty$ and $\gamma$ is compact, a contradiction. 
		\end{proof}
		\begin{claim 3}
			$x_0$ and $y_0$ can not be connected by a curve on $L_0$ everywhere transversal to the flow lines in $L_0$. 
		\end{claim 3}
		
		\noindent
		\begin{proof}
			
			Suppose that there exists a line segment $\sigma$ on $L_0$ everywhere transversal to the flow lines on $L_0$ and connecting $x_0$ and $y_0$. By the local product structure of $\FF$ near $\sigma\in L_0$, there should be a segment $\sigma_i$ in $L_i$ connecting $x_i$ and $y_i$ and everywhere transversal to flow lines on $L_i$.
			Up to taking a sub-segment of $\gamma_i$ if necessary and then
			a sub-segment of $\sigma_i$ we may assume that  $\gamma_i$ does not
			intersect the interior of $\sigma_i$. It follows that the union  of
			$\sigma_i$ and $\gamma_i$ bounds a disk $\mathcal{D}_i$ on $L_i$ as their end points are same. All the flow lines which enter  $\mathcal{D}_i$  transversally intersecting $\sigma_i$ have to exit $\mathcal{D}_i$ transversally intersecting $\sigma_i$. Poincar\'{e}-Hopf theorem says that there exists at least one flow line tangent to $\sigma_i$, a contradiction. 
		\end{proof}
		
		By Lemma \ref{hausdorff} the leaf space of the flow foliation in $L_0$
		is homeomorphic to the reals. Hence any two distinct flow lines
		in $L_0$ are connected by a transversal.
		
		This contradiction proves
		Proposition $5.1$. 
	\end{proof}
	
	Now we are ready to prove our final claim:
	\begin{theorem}\label{qg}
		The flow lines are uniformly quasigeodesics in their respective leaves
		of $\FF$.  
	\end{theorem}
	
	\begin{proof}
		We prove the theorem by contradiction. 
		Again we use a Candel metric in $\F$.
		We assume that the geodesics are not uniform quasigeodesic on their leaves. From this assumption we  will construct sequence of pairs 
		$\{(x_i,y_i)\}$ such that $x_i$ and $y_i$ are connected by a flow segment $\gamma_i$ where $\text{ length}(\gamma_i)\rightarrow \infty$ but $d_{L_i}(x_i, y_i)$ is bounded. 
		Here $L_i$ is the $\FF$ leaf containing both $x_i, y_i$.
		But that will contradict the `weak quasigeodesic property' in proposition \ref{wqg}. A very similar result was proved in \cite{FM01}, we reconstruct the same arguments in our specific case.

		By our assumption that flow lines are not uniform quasigeodesics, we get that for any $K>0$, there exists a flow segment $\gamma_{[x,y]}$ on a leaf $L_x$ such that 
		$$\text{length}(\gamma_{[x,y]})/d_{L_{x}}(x,y)>2K\text{ and }\text{length}(\gamma_{[x,y]})>K$$
		
		\noindent
		Here the $x, y, L_x$ depend of the $K$, we omit the explicit dependence.
		Consider the geodesic  $g_x=g_y$ on $L_x$ with ideal points
		$$\eta^{+}(x)\ =\ \eta^{+}(y) \ \ \ {\rm and} \ \ \ 
		\eta^{-}(x)\ =\ \eta^{-}(y) \ \ \ {\rm on} \ \ \  S^1(L_x).$$
		\noindent
		By Lemma \ref{geodesic nbd}, there exists $d>0$ such that $\gamma\subset\mathcal{N}_{d}(g_x)$,
		where the neighborhood is in $L_x$. This $d$ is global.
		Let $\rho:L_{x}\rightarrow g_x$ be the `closest point map', which means $\rho(p)$ is the closest point on $g_x$ from $p\in L_{x}$. This
		is the orthogonal projection in $L_x$ to $g_x$.
		It follows that:
		$$d_{L_x}(\rho(x),\rho(y))\leq d_{L_{x}}(x,y)\leq d_{L_x}(\rho(x),\rho(y))+2d\text{  (*)}$$
		Let us assume that $d_{L_x}(x,y)>1+2d$. Hence $d_{L_{x}}(\rho(x),\rho(y))>  1$ by $(*)$ and 
		$$\frac{\text{ length}(\gamma_{[x,y]})}{d_{L_x}(\rho(x),\rho(y))}\geq\frac{\text{ length}(\gamma_{[x,y]})}{d_{L_x}(x,y)}\geq 2K>K+\frac{K}{d_{L_x}(\rho(x),\rho(y))} $$
		Therefore 
		$$\frac{\text{length}(\gamma_{[x,y]})}{K}>d_{L_x}(\rho(x),\rho(y))+1>\lceil d_{L_x}(\rho(x),\rho(y))\rceil$$
		
		\noindent
		where $\lceil a \rceil$ denotes the integer $n$ such that $n-1<a\leq n$. 
		
		Suppose $n_0=\lceil d_{L_x}(\rho(x),\rho(y))\rceil$, then $\text{length}(\gamma_{[x,y]})>n_0 K$.
		Also 
		$$n_0-1 \ < \ \lceil d_{L_x}(\rho(x),\rho(y))\rceil \ \leq \  n_0,$$
		
		\noindent
		hence we can construct a sequence $\{\rho(x)=z_0,z_1,...,z_n=\rho(y)\}$ of points in $g_x$,  such that $d_{L_x}(z_{i-1},z_i)=1$ for all $i<n_0$ and $d_{L_x}(z_{n_0-1},z_{n_0})\leq 1$. Next we consider the sequence $x=x_0,x_1,...,x_{n_0}$ where $x_i$ is the last point on $\gamma_{[x,y]}$ such that $\rho(x_i)=z_i$.
		
		If $\gamma_i$ denote the flow segment joining $x_{i-1}$ and $x_i$, we have $\gamma_{[x,y]}=\gamma_1 * \gamma_2 *...*\gamma_{n_0}$. 
		Hence 
		$$\sum_{n=1}^{n_0}\text{length}(\gamma_i)=\text{length}(\gamma_{[x,y]})>n_0K$$
		By the pigeonhole principle there exists $x_{i-1}$ and $x_{i}$ such that $\text{ length}(\gamma_{[x_{i-1},x_i]})>K$. But from $(*)$ we get that for all $i$, $$d_{L_x}(x_{i-1},x_{i})\leq d_{L_x}(\rho(x_{i-1},x_i))+2d=d_{L_x}(z_{i-1},z_i))+2d < 1+2d.$$ 
		
		As the choice of $K>0$ was arbitrary, this proves that the `weak quasigeodesic property' in lemma \ref{wqg} is not true for $b=1+2d$, a contradiction. We
		conclude that flow lines are uniformly quasigeodesic on their respective leaves of $\FF$.
		
		This finishes the proof of the theorem \ref{qg}.
	\end{proof}
	
	\noindent
	\textbf{Conclusion:} Section $4$ shows that every leaf in $\F$ is Gromov
	hyperbolic when lifted to the universal cover. Theorem \ref{qg} proves that 
	the flow foliation is a leafwise quasigeodesic subfoliation
	of $\FF$.
	Moreover Proposition \ref{nf} proves that all leaves of $\F$ which are not contained in $\A$ or $\R$ are \textit{non funnel}, whereas all leaves in $\A$ or $\R$ are $funnel$ by corollary \ref{f}. This completes the proof of the Main Theorem \ref{main}. $\square$
	

\end{document}